\documentclass[reqno]{amsart}
\usepackage{amssymb,amsmath}
\usepackage{amsthm}
\usepackage{color,graphicx}
\usepackage{xcolor}
\usepackage{hyperref}

\newcommand{\er}{\mathbb R}
\newcommand{\R}{\mathbb R}

\newcommand{\V}{{\mathcal V}}

\newtheorem{theorem}{Theorem}[section]
\newtheorem{proposition}[theorem]{Proposition}
\newtheorem{remark}[theorem]{Remark}
\newtheorem{lemma}[theorem]{Lemma}
\newtheorem{corollary}[theorem]{Corollary}
\newtheorem{definition}[theorem]{Definition}

\numberwithin{equation}{section}

\begin{document}
\vglue-1cm \hskip1cm
\date{}
\author{Filipe Oliveira}
\author{Ademir Pastor}

\title{On a Schr\"odinger system arizing in nonlinear optics}



\begin{abstract}
	We study the nonlinear Schr\"odinger system 
	\begin{displaymath}
	\left\{\begin{array}{lllll}
	\displaystyle iu_t+\Delta u-u+(\frac{1}{9}|u|^2+2|w|^2)u+\frac{1}{3}\overline{u}^2w=0,\\
	i\displaystyle \sigma w_t+\Delta w-\mu w+(9|w|^2+2|u|^2)w+\frac{1}{9}u^3=0,
	\end{array}\right.
	\end{displaymath}
	for $(x,t)\in \er^n\times\er$, $1\leq n\leq 3$ and $\sigma,\mu>0$. This system models the interaction between a optical beam and its third harmonic in a material with Kerr-type nonlinear response.  We prove the existence of ground state solutions, analyse its stability, and establish local and global well-posedness results as well as several criteria for blow-up.
	
	\medskip
	
			
			\medskip
		
\end{abstract}

\maketitle

\section{Introduction}
In recent years, cascading nonlinear processes have attracted an increasing interest. It is now well understood that this phenomena leads to effective higher-order nonlinearities in materials with $\chi^{(2)}$ and $\chi^{(3)}$ susceptibilities, in particular in the framework of second and third-order generation (see for instance \cite{Optics1},\cite{Optics2},\cite{Optics3},\cite{Optics4},\cite{Optics5} and references therein).
In \cite{Optics6}, Sammut et al. introduced a new model for the resonant interaction between a monochromatic beam with frequency $\omega$ propagating in a Kerr-type medium and its third harmonic (with frequency $3\omega$). The third-harmonic generation leads to features typical of non-Kerr $\chi^{(2)}$ media. We begin by briefly detailing its derivation. For a more thorough explanation of the computations and approximations involved we refer the reader to \cite{Optics7} and \cite{Optics6}. Let $(\vec{E},\vec{B})$ the electromagnetic field, $\mu_0$ and $\epsilon_0$, respectively, the vacuum permeability and permittivity, $c$ the speed of light in the vacuum  and $\vec{D}$ the electric displacement vector. From the Maxwell-Faraday's equation
$$\frac{\partial \vec{B}}{\partial t}=-\vec{\nabla}\times\vec{E}$$
and Amp\`ere's Law (for nonmagnetic materials and in the absence of free currents)
$$\vec{\nabla} \times \vec{B}=\mu_0\frac{\partial \vec{D}}{\partial t},$$
we obtain 
$$\vec{\nabla}\times \vec{\nabla}\times \vec{E}+\mu_0\frac{\partial^2 \vec{D}}{\partial t^2}=0.$$
Using the constitutive law $\vec{D}={\rm n}^2\epsilon_0\vec{E}+4\pi\epsilon_0\vec{P}_{NL}$, where $\vec{P}_{NL}$ is the nonlinear part of the polarization vector and ${\rm n}$ the linear refractive index, the identity $\mu_0\epsilon_0c^2=1$ and noticing that $\vec{\nabla}\times\vec\nabla\times \vec{E}=-\Delta\vec{E}+\vec{\nabla}(\vec{\nabla}\cdot \vec{E})$, we get, after neglecting the last term in this identity, the vectorial wave equation
\begin{equation}
\label{phys1}
\Delta \vec{E}-\frac{{\rm n}^2}{c^2}\frac{\partial^2\vec{E}}{\partial t^2}=\frac{4\pi}{c^2}\frac{\partial^2\vec{P}_{NL}}{\partial t^2}.
\end{equation}
Assuming that the beams propagate in a slab waveguide, in the direction of the $(Oz)$ axis, we decompose one of the transverse directions of $\vec{E}$ in two frequency components as
$$
E=\Re e\Big(E_1e^{i(k_1z-\omega t)}+E_3e^{i(k_3z-3\omega t)}\Big),
$$ 
where $\Re e (Z)$ stands for the real part of the complex number $Z$.
Each one of these frequency components satisfy equation \eqref{phys1} for suitable values of the polarization, namely, $P_{NL}(\omega)e^{-i\omega t}$ and $P_{NL}(3\omega)e^{-3i\omega t}$, where the nonlinear polarization  can be written in terms of the $\chi^{(3)}$ susceptibility as
$$P_{NL}=\chi^{(3)}E^3=\chi^{(3)}\sum_{\omega_j}P_{NL}(\omega_j)e^{-\omega_j t}.$$ 
A simple computation yields
$$P_{NL}(\omega)e^{-i\omega t}=\frac 18\chi^{(3)}(3|E_1|^2E_1+6|E_3|^2E_1+3E_3\overline{E}_1^2e^{-i(3k_1-k_3)z})e^{i(k_1z-\omega t)}$$
and $$P_{NL}(3\omega)e^{-3i\omega t}=\frac 18\chi^{(3)}(6|E_1|^2E_3+3|E_3|^2E_3+E_1^3e^{-i(3k_1-k_3)z})e^{i(k_3z-3\omega t)}.$$

By plugging into \eqref{phys1} the quantities $E_1e^{i(k_1z-\omega t)}$ and $E_3e^{i(k_1z-\omega t)}$, and under the slowly-varying amplitude approximation,  we obtain the system
\begin{displaymath}
\left\{\begin{array}{lllllll}
\displaystyle \Delta_{\perp} E_1+2ik_1 \frac{\partial E_1}{\partial z}+\Big(\frac{({\rm n}(\omega))^2\omega^2}{c^2}-k_1^2\Big)E_1+\chi(|E_1|^2E_1+2|E_3|^2E_1+E_3\overline{E}_1^2e^{-i(3k_1-k_3)z})=0\\
\\
\displaystyle \Delta_{\perp} E_3+2ik_3 \frac{\partial E_3}{\partial z}+\Big(\frac{9({\rm n}(3\omega))^2\omega^2}{c^2}-k_3^2\Big)E_3+9\chi(2|E_1|^2E_3+|E_3|^2E_3+\frac 13E_1^3e^{-i(3k_1-k_3)z})=0,
\end{array}\right.
\end{displaymath}
where $\chi=-\dfrac{3\pi \omega^2\chi^{(3)}}{c^2}$.

\noindent
Using the dispersion relations $k_1^2=\frac{({\rm n}(\omega))^2\omega^2}{c^2}$, $k_3^2=\frac{9({\rm n}(3\omega))^2\omega^2}{c^2}$ and introducing the dimensionless variables $t=z_dz$, $(x_1,x_2)=x_0(x,y)$ for a given beam width $x_0$ with associated diffraction length $z_d=2x_0^2k_1$, this system can be reduced to 
\begin{equation}
\begin{cases}
{\displaystyle
iU_t+\Delta U+\left(\frac{1}{9}|U|^2+2|W|^2\right)U+\frac{1}{3}\overline{U}^2W=0},\\
{\displaystyle i\sigma W_t+\Delta W-\alpha\sigma W+\Big(9|w|^2+2|u|^2\Big)W+\frac{1}{9}U^3=0},
\end{cases}
\end{equation}
where $U=3(k_1x_0\chi)^{\frac 12}E_1$, $W=3(k_1x_0\chi)^{\frac 12}E_3e^{-i(3k_1-k_3)z})$, $\sigma=k_3/k_1$ and $\alpha=2k_1(3k_1-k_3)x_0^2$.\\
Finally, considering the nonlinearity-induced propagation constant $\beta$, and introducing $u$ and $w$ trough the relations
$$
U(x,t)=\sqrt{\beta}e^{i\omega t}u(\sqrt{\beta}x,\sqrt{\beta}t), \quad W(x,t)=\sqrt{\beta}e^{i3\omega t}u(\sqrt{\beta}x,\sqrt{\beta}t),
$$
we get the nonlinear Schr\"odinger system 
\begin{equation}
\label{nlssystem}
\left\{\begin{array}{lllll}
\displaystyle iu_t+\Delta u-u+\left(\frac{1}{9}|u|^2+2|w|^2\right)u+\frac{1}{3}\overline{u}^2w=0,\\
i\displaystyle \sigma w_t+\Delta w-\mu w+\Big(9|w|^2+2|u|^2\Big)w+\frac{1}{9}u^3=0,
\end{array}\right.
\end{equation}
where $\mu=(3+\frac{\alpha}{\beta})\sigma.$
Note that at resonance ($k_3=3k_1$), $\sigma=3$ and $\mu=3\sigma$. This equality will play a major role in several results presented in this paper. 

\bigskip

From a mathematical point of view, the system \eqref{nlssystem} has been studied in \cite{Pastor2} and \cite{Optics6} in one space dimension. In \cite{Pastor2}, the authors established local and global well-posedness results for the associated Initial Value Problem with periodic initial data. Furthermore, they showed the existence of smooth curves of periodic standing-wave solutions (dnoidal waves) and proved several  results concerning their linear and nonlinear stability. In \cite{Optics6}, the linear stability of localized stationary solutions was adressed and some numerical simulations presented.

\bigskip

 In the present paper we are concerned with the study of \eqref{nlssystem} in Euclidean space $(x,t)\in \er^n\times \er$, $1\leq n\leq 3$. Our main goal is to study the Cauchy problem associated with \eqref{nlssystem} in the $L^2$-based Sobolev space of order one, $H^1(\R^n)$, the so-called {\it energy space}. This terminology comes from the fact that such a system conserves, at least in a formal level, the energy functional
\begin{equation}\label{energy}
\begin{split}
E(u,w)=\frac{1}{2}\int\left(| \nabla u|^2+|\nabla w|^2+| u|^2+\mu|w|^2\right)-\int\left(\frac{1}{36}|u|^4+\frac{9}{4}| w|^4+|u|^2| w|^2+\frac{1}{9}\Re e (\overline{u}^3w)\right)\\
\end{split}
\end{equation}
and the mass
\begin{equation}\label{mass}
M(u,w)=\int\left(|u|^2+3\sigma |w|^2\right).
\end{equation}

It is well-known that for Schr\"odinger-type equations with cubic nonlinearities, the space dimension $n=2$ is critical in the sense that global existence in the energy space is guaranteed provided that the initial data has $L^2$ norm below the one of the ground state (see for instance \cite{w0}). Hence, since we are interested in addressing this type of issue for \eqref{nlssystem}, the associated stationary problem must also be studied.  Recall that
standing waves are special solutions of  \eqref{nlssystem}  of the form
\begin{equation}\label{standing}
u(x,t)=e^{i\omega t}P(x), \qquad w(x,t)=e^{3i\omega t}Q(x),
\end{equation}
where $P$ and $Q$ are real functions with a suitable decay at infinity.
By replacing \eqref{standing} into \eqref{nlssystem} we see that $(P,Q)$ must satisfy
\begin{equation}\label{standindsys}
 \begin{cases}
{\displaystyle \Delta P-(\omega+1)P+\left(\frac{1}{9}P^2+2Q^2\right)P+\frac{1}{3}{P}^2Q=0,}\\
{\displaystyle \Delta Q-\Big(\mu+3\sigma\omega) Q+(9Q^2+2P^2\Big)Q+\frac{1}{9}P^3=0.}
\end{cases}
\end{equation}

\bigskip

\bigskip

The rest of this paper is organized as follows: in section 2 we will show the existence of solutions for \eqref{standindsys} and study their properties. By a {\it solution} of \eqref{standindsys} we mean a pair of functions $(P,Q)\in H^1(\R^n)\times H^1(\R^n)$ such that
$$
\int (\nabla P\cdot\nabla f+(\omega+1)Pf=\int \left(\frac{1}{9}P^3+2Q^2P+\frac{1}{3}{P}^2Q\right)f
$$
and
$$
\int (\nabla Q\cdot\nabla g+(\mu+3\sigma\omega) Qg=\int \left(9Q^3+2P^2Q+\frac{1}{9}P^3\right)g,
$$
for any pair $(f,g)\in H^1(\R^n)\times H^1(\R^n)$. So, a solution is {\it a priori} understood in the weak sense. However, as it is standard from the elliptic regularity theory, such a weak solution is indeed a strong solution in the usual sense (see, for instance, \cite{gil}). It is easy to check that solutions of \eqref{standindsys}, also called {\it bound states}, are the critical points of the action functional defined by
\begin{equation}
\label{action}
S(P,Q):=E(P,Q)+\frac{\omega}{2} M(P,Q),
\end{equation}
that is, denoting by $\mathcal{B}=\mathcal{B}(\omega,\mu,\sigma)$ the set of all solutions of \eqref{standindsys}, we have 
$$
\mathcal{B}(\omega,\mu,\sigma):=\{(P,Q)\in H^1\times H^1\,:\,S'(P,Q)=0\}.
$$
\noindent
Among all bound states, we will single out the {\it ground states}, i.e., the bound states which minimize the action $S$ among all other bound states. We will prove that such a set of solutions is indeed nonempty (Theorem \ref{existenceGS}). The method we use to prove this result is a variational one, by minimizing $S$ in the so-called {\it Nehari manifold}. In addition, we also study when a ground state has both components nontrivial.

\bigskip

\noindent
In Section 3 we study the Cauchy problem associated to \eqref{nlssystem} for initial data in the energy space $(u_0,w_0)\in H^1(\er^n)\times H^1(\er^n)$. After establishing local-well posedness and a blow-up alternative (Theorem \ref{localtheorem}) we show that the Cauchy problem is globally well-posed in dimension $n=1$ (Corollary \ref{cor32}). In what concerns dimensions $n=2$ and $n=3$, we will give sufficient conditions for global well-posedness in terms of the size of the initial data with respect to the size of ground states at resonance $\mu=3\sigma$ (Theorems \ref{globaln=2} and \ref{globaln=3}).

\bigskip

\noindent
In Section 4 we study the blow-up of solutions to \eqref{nlssystem}. We will begin by showing in Theorem \ref{sharpn=2} that, at resonance, Theorem \ref{globaln=2} is sharp. In dimension $n=3$, we also show that Theorem \ref{globaln=3} is sharp at resonance provided that the initial data $(u_0,w_0)$ lies in $\mathbb{H}=H^1(\R^3)\cap L^2(\R^n,|x|^2dx)$ (Theorem \ref{sharpn=3}). Moreover, we exhibit several conditions implying that the solution blows up either forward or backward in time (Theorems \ref{T47} and \ref{T48}).

\bigskip

\noindent
Finally, in Section 5, we deal with the stability/instability of the ground states $(P,Q)$. We will show that the ground states are orbitally stable in dimension one provided $\omega+1=\mu+3\sigma\omega$ (Theorem \ref{teoesta}). On the other hand, we prove that ground states are unstable if either   $n=3$ and $\mu>0$ or $n=2$ and $\mu\neq3\sigma$.

\bigskip

\noindent
Throughout the paper we will use standard notation in PDEs. Unless otherwise stated, the domain of the different integrals is $\R^n$, hence, for convenience, we will  denote $\int_{\R^n}fdx$ simply by $\int f$. Also, $C$ will represent a generic constant which may vary from inequality to inequality.

\section{Existence of ground states} \label{secexis}

The main goal of this section is to prove the existence of ground states.  More precisely, we will establish the following result:

\begin{theorem}\label{existenceGS}
	Let $1\leq n \leq 3$, $\sigma,\mu>0$ and $\omega>\max\{-1,-\mu/3\sigma\}$. Then the set of ground states, denoted by $\mathcal{G}(\omega,\mu,\sigma)$, is nonempty, that is,
	$$\mathcal{G}(\omega,\mu,\sigma):=\Big\{(P_0,Q_0)\in \mathcal{B}\setminus\{(0,0)\}\,:\, S(P_0,Q_0)\leq S(P,Q), \forall (P,Q)\in \mathcal{B},\Big\}\neq\emptyset.
	$$
	In addition, there exists at least one ground state, say, $(P_0,Q_0)$, which is radially symmetric, $Q_0$ is positive and $P_0$ is either positive or identically zero.
\end{theorem}

 Before proceeding, let us establish some Pohojaev-type identities for the solutions of \eqref{standindsys}, which will be useful later.

\begin{lemma}\label{pohojaevlemma}
	Assume that \eqref{standindsys} has a solution $(P,Q)\in H^1(\R^n)\times H^1(\R^n)$. Then the following identities hold:
	\begin{equation}\label{poha1}
	\int \left(-|\nabla P|^2-(\omega+1)P^2+\frac{1}{9}P^4+2P^2Q^2+\frac{1}{3}P^3Q \right)=0,
	\end{equation}
	\begin{equation}\label{poha2}
	\int \left(-|\nabla Q|^2-(\mu+3\sigma\omega)Q^2+9Q^4+2P^2Q^2+\frac{1}{9}P^3Q \right)=0,
	\end{equation}
	and
	\begin{equation}\label{poha3}
	(n-4)\int \left(|\nabla P|^2+|\nabla Q|^2\right) +n(\omega+1)\int P^2+n(\mu+3\sigma\omega)\int Q^2 =0.
	\end{equation}
\end{lemma}
\begin{proof}
	By multiplying the first equation in \eqref{standindsys} by $P$, the second one by $Q$, integrating over $\R^n$ and using integration by parts, we obtain \eqref{poha1} and \eqref{poha2}.\\
		On the other hand, by the same procedure but multiplying this time the two equations by $x\cdot \nabla P$ and $x\cdot \nabla Q$ respectively, we deduce 
	\begin{equation}\label{poha4}
	\int\left(\frac{(n-2)}{2} |\nabla P|^2 +\frac{n(\omega+1)}{2} P^2-\frac{n}{36}P^4+2Q^2Px\cdot\nabla P+\frac{1}{3}P^2Qx\cdot\nabla P \right)=0
	\end{equation}
	and
	\begin{equation}\label{poha5}
	\int\left(\frac{(n-2)}{2} |\nabla Q|^2 +\frac{n(\mu+3\sigma\omega)}{2} Q^2-\frac{9n}{4}Q^4+2P^2Qx\cdot\nabla Q+\frac{1}{9}P^3x\cdot\nabla Q \right)=0.
	\end{equation}
	Now, integration by parts yields
	$$
	\int \left(2P^2Qx\cdot\nabla Q+\frac{1}{9}P^3x\cdot\nabla Q \right)=-\int \left(2Q^2Px\cdot\nabla P+\frac{1}{3}P^2Qx\cdot\nabla P+nP^2Q^2+\frac{n}{9}P^3Q \right).
	$$
	By replacing this last identity into \eqref{poha5} and summing the resulting equation with \eqref{poha4},
	\begin{equation}\label{poha6}
	\begin{split}
	\frac{(n-2)}{2}\int(|\nabla P|^2+& |\nabla Q|^2) + \frac{n(\omega+1)}{2}\int P^2+\frac{n(\mu+3\sigma\omega)}{2}\int Q^2\\
	&-\frac{n}{4}\int\left(\frac{1}{9}P^4+9Q^4+4P^2Q^2+\frac{4}{9}P^3Q \right)=0.
	\end{split}
	\end{equation}
	Also, summing equations \eqref{poha1} and \eqref{poha2}, we obtain
	\begin{equation}\label{poha7}
	\int\left(\frac{1}{9}P^4+9Q^4+4P^2Q^2+\frac{4}{9}P^3Q \right)=\int (|\nabla P|^2+ |\nabla Q|^2)+\int\left( (\omega+1)P^2+(\mu+3\sigma\omega)Q^2\right).
	\end{equation}
	Identity \eqref{poha3} then follows by combining \eqref{poha7} and \eqref{poha6}.
	\end{proof}

\begin{remark}
As an immediate consequence of Lemma \ref{pohojaevlemma} we see that, under the assumption $\omega>\max\{-1,-\mu/3\sigma\}$, ground state solutions in $H^1(\er^n)\cap L^4(\er^n)$ do not exist if $n\geq4$.
\end{remark}

In order to prove Theorem \ref{existenceGS}, we will study a minimization problem in the Nehari manifold.

\begin{lemma}\label{equivalents}
Let $$\mathcal{N}:=\{(u,w)\in H^1(\R^n)\times H^1(\R^n)\,: (u,v)\neq (0,0), S'(u,w)\perp (u,w)\}$$
be the Nehari manifold associated to the action $S$.
Then any solution of the minimization problem 
\begin{equation}
\label{minimization1}
\inf \{S(u,w)\,:\,(u,w)\in \mathcal{N}\},
\end{equation}
is a ground state.
\end{lemma}
\begin{proof}
Since $\mathcal{B}\subset\mathcal{N}$, it is enough to prove that all critical points of \eqref{minimization1} are indeed bound states.\\
We begin by noticing that $(u,w)\in \mathcal N$ if and only if $(u,w)\neq (0,0)$ and
\begin{equation} \label{deftau}
\tau(u,w):=\int |\nabla u|^2+|\nabla w|^2+(1+\omega)u^2+(\mu+3\sigma\omega)w^2-\frac 19u^4-4u^2w^2-9w^4-\frac 49u^3w=0.
\end{equation}
Furthermore,
$$\langle \tau'(u,w),(u,w) \rangle_{L^2}=2\Big(\int |\nabla u|^2+|\nabla w|^2+(1+\omega)u^2+(\mu+3\sigma\omega)w^2-\frac 29u^4-8u^2w^2-18w^4-\frac 89u^3w\Big),$$
and, if $(u,w)\in\mathcal{N}$, 
\begin{equation}
\label{estmanifold}
\langle \tau'(u,w),(u,w) \rangle_{L^2}=-2\Big(\int |\nabla u|^2+|\nabla w|^2 +(1+\omega)u^2+(\mu+3\sigma\omega)w^2\Big)\neq 0,
\end{equation}
which shows that $\mathcal{N}$ is locally smooth.\\
 In addition, it is easy to check that $[h_1,h_2]\textrm{Hess}\,\tau_{(0,0)}\,^t[h_1,h_2]>0$ for all $(h_1,h_2)\neq(0,0)$, which means that $(0,0)$ is a strict minimizer of $\tau$, hence an isolated point of the set $\{\tau(u,w)=0\}$, implying that $\mathcal{N}$ is a complete manifold. Finally, any critical point of $S$ constrained to $\mathcal{N}$ is a (unconstrained) critical point of $S$. Indeed, let us consider $(u_0,w_0)\in\mathcal{N}$ a critical point of $S$ constrained to $\mathcal{N}$. There exists a Lagrange multiplier $\lambda$ such that $S'(u_0,w_0)=\lambda \tau'(u_0,w_0).$
By taking the $L^2$ scalar product with $(u_0,w_0)$, $$\langle S'(u_0,w_0),(u_0,w_0)\rangle_{L^2}=\lambda \langle\tau'(u_0,w_0),(u_0,w_0)\rangle_{L^2},$$ that is, in view of (\ref{estmanifold}), $0=-2\lambda\Big(\int |\nabla u_0|^2+|\nabla w_0|^2+(1+\omega)u_0^2+(\mu+3\sigma\omega)w_0^2\Big)$. Hence $\lambda=0$ and $S'(u_0,w_0)=0$, which establishes the claim.
\end{proof}

As a consequence of Lemma \ref{equivalents}, in order to show Theorem \ref{existenceGS} we will prove the existence of a minimizer to problem \eqref{minimization1}.

\begin{proof}[Proof of Theorem \ref{existenceGS}]

Notice that for $(u,w)\in H^1\times H^1$, $(u,w)\neq (0,0)$, with $\tau(u,w)\leq 0$, there exists $t\in]0,1]$ such that $(tu,tw)\in\mathcal{N}$. Indeed, if $\tau(u,w)=0$, one chooses $t=1$. If $\tau(u,w)<0$ we simply observe that 
\[
\begin{split}
\tau(tu,tw)&=t^2\Big\{\int \Big[|\nabla u|^2+|\nabla w|^2+(1+\omega)u^2+(\mu+3\sigma\omega)w^2\\
&\quad\quad \quad\quad-t^2\Big(\frac 19u^4+4u^2w^2+9w^4+\frac 49u^3w\Big)\Big]\Big\}:=t^2T_{u,w}(t),
\end{split}
\]
with $T_{u,w}(0)>0$ and $T_{u,w}(1)<0$. The Intermediate Value Theorem allows us to conclude.

\medskip

\noindent
We now take a minimizing sequence $(u_j,w_j)\in \mathcal{N}$ for the problem
$$m=\inf\{S(u,w):\,(u,w)\in\mathcal{N}\}.$$
Since $(u_j,w_j)\in\mathcal{N}$,
$$
S(u_j,w_j)=\frac 14\Big(\int |\nabla u_j|^2+|\nabla w_j|^2+(1+\omega)u_j^2+(\mu+3\sigma\omega)w_j^2\Big),
$$
hence it is clear that  $m\geq 0$ and that  $(u_j,w_j)$ is bounded in $H^1\times H^1$.

We put $u_j^*$ and $v_j^*$ the decreasing radial rearrangements of $|u_j|$ and $|v_j|$, respectively. It is well-known that this rearrangement preserves the $L^p$ norm ($1\leq p \leq +\infty$).
Furthermore, the P\'olya-Szeg\"o inequality,
$$\|\nabla f^*\|_{L^2}\leq \|\nabla |f|\|_{L^2},$$ in addition with the inequality $\|\nabla |f|\|_{L^2}\leq \|\nabla f\|_{L^2}$ (see \cite{Lions1}) shows that
$$
S(u_j^*,v_j^*)\leq S(u_j,v_j).
$$
On the other hand, the Hardy-Littlewood inequality,
$$\int |uw|\leq \int u^*w^*,$$
combined with the monotonicity of the map $\lambda\mapsto\lambda^4$ (see for instance \cite{Rear} for details) yields 
$$\int u^2w^2\leq \int (u^*)^2(w^*)^2\,\,\textrm{ and }\,\,\int |u^3w|\leq \int (u^*)^3w^*.$$A combination of these inequalities give
$$
\tau(u_j^*,w_j^*)\leq \tau(|u_j|,|w_j|)\leq \tau (u_j,w_j)=0.
$$
Next, let $t_j\in]0,1]$ be such that $(t_ju_j^*,t_jw_j^*)\in \mathcal{N}.$
We have
$$
S(t_ju_j^*,t_jw_j^*)=t_j^2S(u_j^*,w_j^*)\leq S(u_j^*,w_j^*)
$$ 
and hence, we obtained a minimizing sequence $(t_ju_j^*,t_jv_j^*)$ of radially decreasing functions, denoted again, in what follows, by $(u_j,v_j)$.
Since this sequence is bounded in $H^1\times H^1$, up to a subsequence, $(u_j,v_j)\rightharpoonup (u_*,v_*)$ weakly in $H^1\times H^1$.

To obtain a convergence in a strong topology, it is often necessary to treat the unidimensional $n=1$ separately due to the lack of compactness of the injection $H^1_d(\R)\hookrightarrow L^4(\R)$, where $H_d^1(\R)$ denotes the space of the radially symmetric functions of  $H^1(\R)$. This lack of compactness is, in a sense, a consequence of the inequality
\begin{equation}
\label{ineq}
|u(x)|\leq C|x|^{\frac{1-n}2}\|u\|_{H^1(\R^n)}
\end{equation}
for $u\in H_d^1(\R^2)$, which provides no decay in the case $n=1$. However, if $u$ is also radially decreasing, it is easy to establish that
$$|u(x)|\leq C|x|^{-\frac{n}2}\|u\|_{L^2(\R^n)},$$
which provides decay in all space dimensions, hence compactness by applying the classical Strauss' compactness lemma (\cite{Strauss}). Therefore, putting 
$$
H_{rd}^1(\R^n)=\{u \in H^1_d(\R^n)\,:\,u\textrm{ is radially decreasing}\},
$$ we get the compactness of the injection $H_{rd}^1(\R^n)\hookrightarrow L^4(\R^n)$ for all $n\geq 1$ (see the Appendix of \cite{BL} or Section 1.7 in \cite{Cazenave} for more details). Consequently, up to a subsequence, $(u_j,v_j)\to (u_*,v_*)$ strongly in $L^{4}$ and almost everywhere. In particular this shows that $(u_*,v_*)$ is radially symmetric and nonnegative.

Next, since 
$$
\int\frac{1}{36}u_j^4+\frac{9}{4}w_j^4+u_j^2w_j^2+\frac{1}{9}u_j^3w_j\to \int\frac{1}{36}u^4+\frac{9}{4}w^4+u^2w^2+\frac{1}{9}u^3w,
$$ we deduce that
$$\tau(u_*,w_*)\leq \liminf \tau(u_j,w_j)=0.$$
Once again, let $t\in]0,1]$ such that $(tu_*,tv_*)\in \mathcal{N}$. Thus,
$$
m\leq S(tu_*,tw_*)=t^2S(u_*,w_*)\leq\liminf S(u_j,v_j)=m.
$$
This implies that $(tu_*,tw_*)$ is a minimizer. In particular, all inequalities above are in fact equalities, which means that $t=1$, $(u_*,w_*)\in \mathcal{N}$ and $(u_j,w_j)\to (u_*,w_*)$ strongly in $H^1$.\\
Finally, it is easy to see that $(P_0,Q_0)=(u_*,w_*)$ is a ground state accordingly to the conclusions of the theorem. Indeed, by elliptic regularity $(P_0,Q_0)$ is a $C^2$ solution and satisfies
\begin{equation*}
\begin{cases}
\Delta P_0-(\omega+1)P_0=-(\frac{1}{9}P_0^2+2Q_0^2)P_0-\frac{1}{3}{P}_0^2Q_0 \leq0,\\
\Delta Q_0-(\mu+3\sigma\omega) Q_0=-(9Q_0^2+2P_0^2)Q_0-\frac{1}{9}P_0^3 \leq0.
\end{cases}
\end{equation*}
Therefore, from the maximum principle (see, for example, Theorem 3.5 in \cite{gil}) both $P_0$ and $Q_0$ are either positive or identically zero. Note that $Q_0$ is not identically zero; otherwise so is $P_0$. This completes the proof of Theorem \ref{existenceGS}.  
\end{proof}

Next we will pay particular attention to the question of when both components of a ground state are non-trivial. First of all, recall that a ground state of the scalar equation
\begin{equation}\label{nlsg}
\Delta w-(\mu+3\sigma\omega) w+9w^3=0,
\end{equation}
is a solution (in the weak sense) that minimizes the action $S_0(w):=S(0,w)$ among all solutions os \eqref{nlsg}.  As is well known (see, for instance, \cite{BL} or \cite{Cazenave}), for $\mu+3\sigma \omega>0$, \eqref{nlsg} has a unique (up to translation) ground state which is positive, radially symmetric and decays exponentially at infinity. 

It is easily seen that if $(0,Q)$ is a ground state of \eqref{standindsys} then $Q$ is a ground state of \eqref{nlsg}. Thus, a natural question is if the reciprocal is also true, that is, if  $Q$ is a ground state of \eqref{nlsg}, is it true that  $(0,Q)$ is a ground state of \eqref{standindsys}? As we will see below,  depending on the parameters $\mu$ and $\sigma$, the answer to this question may be negative or positive:

\begin{proposition}
In addition to the assumptions of Theorem \ref{existenceGS},  assume $\mu=3\sigma$ and $\mu\geq 9^{\frac{4}{4-n}}$. Then there exists a pair $(P^*,Q^*)$ in the Nehari manifold $\mathcal{N}$ such that 
$$
S(P^*,Q^*)<S(0,Q),
$$
where $Q$ is the ground state of \eqref{nlsg}. In particular $(0,Q)$ is not a ground state of \eqref{standindsys}.
\end{proposition}
\begin{proof}
In what follows, for real functions $u,w\in H^1$, we introduce the functional
\begin{equation}\label{gal2}
N(u,w):=\int\left(\frac{1}{36}u^4+\frac{9}{4} w^4+u^2 w^2+\frac{1}{9}{u}^3w\right).
\end{equation}
and
\begin{equation}
K(u,w)=\|\nabla u\|_{L^2}^2+\|\nabla w\|_{L^2}^2.
\end{equation}

According to Lemma \ref{equivalents} it suffices to prove the existence of $\theta,t\in\R$ and $W\in H^1$ such that $(t\theta W,tQ)\in \mathcal{N}$ and $S(t\theta W,tQ)<S(0,Q)$. But from the proof of Lemma \ref{equivalents} we have $(t\theta W,tQ)\in \mathcal{N}$ if and only if $\tau(t\theta W,tQ)=0$, where $\tau$ is defined in \eqref{deftau}. Since
$$
\tau(t\theta W,tQ)=K(t\theta W,tQ)+(1+\omega)M(t\theta W,tQ)-4N(t\theta W,tQ),
$$
by taking $t\in\R$ satisfying
\begin{equation}\label{deft}
t^2=\frac{K(\theta W,Q)+(1+\omega)M(\theta W,Q)}{4N(\theta W,Q)}
\end{equation}
we see that $\tau(t\theta W,tQ)=0$ (we will choose $\theta>0$ and $W>0$, so that $N(\theta W,Q)>0$). Consequently, from this point on, we take $t$ as in \eqref{deft}.\\

\medskip

\noindent
Now, in view of the identity,
$$
K(t\theta W,tQ)+(1+\omega)M(t\theta W,tQ)=4N(t\theta W,tQ)
$$
and \eqref{deft}, we deduce
\[
\begin{split}
S(t\theta W,tQ)&= \frac{1}{2}\Big(K(t\theta W,tQ)+(1+\omega)M(t\theta W,tQ)\Big)-N(t\theta W,tQ))\\
& =\frac{1}{4}\Big(K(t\theta W,tQ)+(1+\omega)M(t\theta W,tQ)\Big)\\
&=\frac{t^2}{4}\Big(K(\theta W,Q)+(1+\omega)M(\theta W,Q)\Big)\\
&=\frac{\Big(K(\theta W,Q)+(1+\omega)M(\theta W,Q)\Big)^2}{16N(\theta W,Q)}.
\end{split}
\]
Thus $S(t\theta W,tQ)<S(0,Q)$ if and only if
\begin{equation} \label{equipol}
\Big(K(\theta W,Q)+(1+\omega)M(\theta W,Q)\Big)^2<4N(\theta W,Q)\Big(K(0,Q)+(\omega+1)M(0,Q)\Big),
\end{equation}
where we used that $S(0,Q)=(K(0,Q)+(\omega+1)M(0,Q))/4$. Both sides of \eqref{equipol} are polynomials of degree four in $\theta$. The leading coefficient of the polynomial in the left-hand side is $(K(W,0)+(\omega+1)M(W,0))^2$ whereas the leading coefficient of the polynomial in the right-hand side is 
$$
\frac{1}{9}\left(\int W^4\right)\Big(K(0,Q)+(\omega+1)M(0,Q)\Big).
$$
Therefore, \eqref{equipol} holds, for $\theta$ sufficient large, provided that
\begin{equation}\label{equiW}
(K(W,0)+(\omega+1)M(W,0))^2<\frac{1}{9}\left(\int W^4\right)\Big(K(0,Q)+(\omega+1)M(0,Q)\Big).
\end{equation}
So, we are left to show that \eqref{equiW} holds for some $W\in H^1$. For that, assume  $W(x)=Q(\lambda x)$ for some $\lambda\in\R$ to be determined. With this definition, \eqref{equiW} is equivalent to
$$
\lambda^2\int |\nabla Q|^2+(\omega+1)\int Q^2<\frac{\lambda^{n/2}}{3}\Big(K(0,Q)+(\omega+1)M(0,Q)\Big)^{1/2}\left(\int Q^4\right)^{1/2}.
$$
In view of \eqref{poha3},
\begin{equation}\label{a36}
K(0,Q)=\int|\nabla Q|^2=\frac{n\mu(\omega+1)}{4-n}\int Q^2,
\end{equation}
Also, by using \eqref{poha2} and \eqref{a36}, we deduce
\begin{equation} \label{a37}
\int Q^4=\frac{4}{9}\frac{\mu(\omega+1)}{4-n}\int Q^2.
\end{equation}
By replacing \eqref{a36} and \eqref{a37} into \eqref{equiW}, we then obtain that  \eqref{equiW} is equivalent to
\begin{equation}\label{defff}
\frac{n\mu}{4-n}\lambda^2+1-\frac{4\mu}{9(4-n)}\lambda^{n/2}<0.
\end{equation}
Let $f(\lambda)$ denotes the left-hand side of \eqref{defff}. It is easy to see that such a function has a global minimum at the point $\lambda_0=9^{-2/(4-n)}$. In addition, $f(\lambda_0)=1-\mu\lambda_0^2$. Finally, under the assumption $f(\lambda_0)<0$, which means to say $\mu\geq 9^{4/(4-n)}$, we then see that \eqref{defff} holds for $\lambda=\lambda_0$ and   the proof of the proposition is complete.
\end{proof}

Next, we shall show that under the condition $\omega+1=\mu+3\sigma\omega$, the ground states of \eqref{standindsys} are precisely of the form $(0,Q)$, where $Q$ is a ground state of \eqref{nlsg}. We will closely follow the strategy in \cite{correia}.
Define the functionals
\begin{equation}\label{Idef}
I(u,w)=\int (|\nabla u|^2+ |\nabla w|^2)+\int\left( (\omega+1)u^2+(\mu+3\sigma\omega)w^2\right),
\end{equation}
\begin{equation}\label{Ntil}
\widetilde{N}(u,w)=\frac{1}{4}N(u,w)=\int\left(\frac{1}{9}u^4+9w^4+4u^2w^2+\frac{4}{9}u^3w \right)
\end{equation}
and, for $\lambda>0$, consider the minimization problem
\begin{equation}\label{Ilam}
I_\lambda=\inf \{I(f,g): \;(f,g)\in H^1\times H^1 \;\mbox{with}\; \widetilde{N}(f,g)=\lambda \}.
\end{equation}

Our goal will be to prove that for a certain specific $\lambda$ such a infimum is attained by the ground states of \eqref{standindsys}. Initially, note that, from the homogeneity of $I$ and $\widetilde{N}$, if follows that
\begin{equation}\label{I1prop}
I_\lambda=\lambda^{1/2}I_1.
\end{equation}
Also, from Young and Gagliardo-Nirenberg's inequality,
\[
\begin{split}
\widetilde{N}(u,w)\leq C(\|u\|_{L^4}^4+\|w\|_{L^4}^4)\leq CK(u,w)^{n/2}M(u,w)^{2-n/2}\leq CI(u,w)^2,
\end{split}
\]
which implies that $I_\lambda>0$, for any $\lambda>0$. To motivate which $\lambda$ would be the correct one, we recall that if $(u,w)\in \mathcal{G}(\omega,\mu,\sigma)$ then, by \eqref{poha7}, $\widetilde{N}(u,w)=I(u,w)$. Hence, we must choose $\lambda$ such that $I_\lambda=\lambda$. In view of \eqref{I1prop}, we must choose $\lambda=\lambda_1$, where
\begin{equation}\label{lam1}
\lambda_1:=(I_1)^2.
\end{equation}

\begin{lemma}\label{mlambda}
Let assumptions of Theorem \ref{existenceGS} hold and let $m=\inf\{S(u,w):\,(u,w)\in\mathcal{N}\}$. Then
$$
\lambda_1=4m.
$$
\end{lemma}
\begin{proof}
From the proof of Theorem \ref{existenceGS} we already know the minimization problem \eqref{minimization1} has a solution (a ground state). So, we may fix $(u,w)\in H^1\times H^1$ satisfying $m=S(u,w)$. Since $(u,w)\in \mathcal{G}(\omega,\mu,\sigma)$, we have $\widetilde{N}(u,w)=I(u,w)$ and
\begin{equation} \label{mequi}
m=S(u,w)=E(u,w)+\frac{\omega}{2}M(u,w)=\frac{1}{2}I(u,w)-\frac{1}{4}\widetilde{N}(u,w)=\frac{1}{4}I(u,w).
\end{equation}
Hence, $I(u,w)=4m$. Next, define $(U,W)=(1/4m)^{1/4}(u,w)$. Then, $\widetilde{N}(U,W)=1$ and
$$
I(U,W)=\left(\frac{1}{4m}\right)^{1/2}I(u,w)=(4m)^{1/2}.
$$
This identity implies that $I_1\leq (4m)^{1/2}$, which yields $\lambda_1\leq 4m$.\\
We shall have established the lemma if we prove that $\lambda_1\geq 4m$, that is, $I_1\geq (4m)^{1/2}$. Take any $(z,v)\in H^1\times H^1$ with $\widetilde{N}(z,v)=1$. It then suffices to prove that $(4m)^{1/2}\leq I(z,v)$ or, which is the same, $2S(u,w)^{1/2}\leq I(z,v)$. To prove this, define $(Z,V)=(4m)^{1/4}(z,v)$. It is easy to see that $\widetilde{N}(Z,V)=4m$, $I(Z,V)=(4m)^{1/2}I(z,v)$, and
$$
S(u,w)\leq S(Z,V)=\frac{1}{2}I(Z,V)-\frac{1}{4}\widetilde{N}(Z,V)=\frac{1}{2}(4m)^{1/2}I(z,v)-m=S(u,w)^{1/2}I(z,v)-S(u,w).
$$
The assertion clearly follows from the last inequality.
\end{proof}

\medskip

\noindent
Next, we show the following:
\begin{proposition}\label{varcar}
Under the assumptions of Theorem \ref{existenceGS}, $(u,w)\in \mathcal{G}(\omega,\mu,\sigma)$ if and only if $I(u,w)=I_{\lambda_1}$ and $\widetilde{N}(u,w)=\lambda_1$.\\
In particular,  the set of solutions of the minimization problem \eqref{Ilam} with $\lambda=\lambda_1$ is nonempty.
\end{proposition}
\begin{proof}
Let us first take  $(u,w)\in \mathcal{G}(\omega,\mu,\sigma)$. By reasoning as in \eqref{mequi} and using    Lemma \ref{mlambda}, we get
$$
I(u,w)=4m=\lambda_1=I_{\lambda_1} \quad \mbox{and} \quad \widetilde{N}(u,w)=I(u,w)=4m=\lambda_1,
$$
which shows one of the assertions.\\
Let us now assume that $(u,w)$ satisfies  $I(u,w)=I_{\lambda_1}$ and $\widetilde{N}(u,w)=\lambda_1$. By the Lagrange multiplier theorem, there exists $\eta\in\R$ such that, for any $(f,g)\in H^1\times H^1$,
$$
\int (\nabla u\cdot\nabla f+(\omega+1)uf=2\eta\int \left(\frac{1}{9}u^3+2w^2u+\frac{1}{3}{u}^2w\right)f,
$$
$$
\int (\nabla w\cdot\nabla g+(\mu+3\sigma\omega) wg=2\eta\int \left(9w^3+2u^2w+\frac{1}{9}u^3\right)g.
$$
By taking $(f,g)=(u,w)$, and adding the last two identities, we deduce that $I(u,w)=2\eta\widetilde{N}(u,w)$. But, from 
$$
\lambda_1^{1/2}I_1=I_{\lambda_1}=I(u,w)=2\eta \widetilde{N}(u,w)=2\eta\lambda_1,
$$
we obtain $I_1=2\eta\lambda_1^{1/2}$, which compared to \eqref{lam1} gives $2\eta=1$. Consequently, $(u,w)\in \mathcal{B}(\omega,\mu,\sigma)$ and $I(u,w)=\widetilde{N}(u,w)$.

It remains to show that $(u,w)$ is indeed a ground state. To do so, take any $(z,v)$ in $\mathcal{B}(\omega,\mu,\sigma)$ and let $\kappa:=\widetilde{N}(z,v)>0$. Recalling \eqref{poha7}, we then have $I(z,v)=\widetilde{N}(z,v)=\kappa$ and, 
$$
S(z,v)=\frac{1}{2}I(z,v)-\frac{1}{4}\widetilde{N}(z,v)=\frac{1}{4}I(z,v)=\frac{\kappa}{4}.
$$ 
Define $(\widetilde{z},\widetilde{v})=(\lambda_1/\kappa)^{1/4}(z,v)$. Then, 
$$
\widetilde{N}(\widetilde{z},\widetilde{v})=\frac{\lambda_1}{\kappa}\widetilde{N}(z,v)=\lambda_1
$$
and
$$
\lambda_1^{1/2}I_1=I(u,w)\leq I(\widetilde{z},\widetilde{v})=\left(\frac{\lambda_1}{\kappa}\right)^{1/2}I(z,v)=\left(\frac{\lambda_1}{\kappa}\right)^{1/2}\kappa=\lambda_1^{1/2}\kappa^{1/2}.
$$
This last inequality implies that $\kappa\geq (I_1)^2=\lambda_1$. Thus,
\begin{equation}\label{l34}
S(z,v)=\frac{\kappa}{4}\geq\frac{\lambda_1}{4}=S(u,w),
\end{equation}
which proves that $(u,w)\in \mathcal{G}(\omega,\mu,\sigma)$.
\end{proof}

\bigskip

Finally, we prove the previously announced result:
\begin{proposition}\label{carground}
In addition to the assumptions of Theorem \ref{existenceGS}, suppose that $\omega+1=\mu+3\sigma\omega$. If $(u,w)\in \mathcal{G}(\omega, \mu,\sigma)$ then $u\equiv0$ and $w$ is a ground state of \eqref{nlsg}.\\
In particular, up to translation, ground states are unique.
\end{proposition}
\begin{proof}
Take $(u,w)$ in $\mathcal{G}(\omega, \mu,\sigma)$. From Proposition \ref{varcar} we have $I(u,w)=I_{\lambda_1}$ and $\widetilde{N}(u,w)=\lambda_1$. Let us introduce the function $F:\R^2\to\R$ by $F(x,y)=\frac{1}{9}x^4+9y^4+4x^2y^2+\frac{4}{9}x^3y$. It is easily seen that, restricted to the unit circle $\mathbb{S}^1$, $F$ has two maximum points, namely, $(0,1)$ and $(0,-1)$. In addition, its maximum value is $F(0,\pm1)=9$.

Now, define $U(x):=|(u(x),w(x))|=\sqrt{u(x)^2+w(x)^2}>0$. Thus,
\begin{equation} \label{div1}
\begin{split}
\widetilde{N}(u,w)&=\int F(u(x),w(x))=\int F\left(\frac{1}{U(x)}(u(x),w(x)) \right)U(x)^4\leq \int F(0,1)U(x)^4\\
&=\int F(0,U(x))=\widetilde{N}(0,U).
\end{split}
\end{equation}
Also, because $|\nabla U|^2\leq |\nabla u|^2+|\nabla w|^2$,
\[
\begin{split}
I(0,U)&=\int |\nabla U(x)|^2+(\mu+3\sigma\omega)|U(x)|^2\\
&\leq \int |\nabla u(x)|^2+|\nabla w(x)|^2+(\omega+1)u(x)^2+(\mu+3\sigma\omega)w(x)^2=I(u,w),
\end{split}
\]
where we used that $\omega+1=\mu+3\sigma\omega$. In view of \eqref{div1} and the homogeneity of $\widetilde{N}$, there exists $0<t\leq 1$ such that $\widetilde{N}(0,tU)=\widetilde{N}(u,w)=\lambda_1$. Hence,
$$
I(0,tU)=t^2I(0,U)\leq I(0,U)\leq I(u,w).
$$
By recalling that $(u,w)$ is a minimum $I$ restricted to $\widetilde{N}=\lambda_1$, it must be the case that $t=1$. Thus,
$$
\widetilde{N}(0,U)=\lambda_1 \qquad \mbox{and} \qquad I(0,U)=I(u,w)=I_{\lambda_1}.
$$
Another application of Proposition \ref{varcar} yields that $(0,U)\in \mathcal{G}(\omega, \mu,\sigma)$. Consequently, $U$ must be a ground state of \eqref{nlsg}.

By defining $(z,v)=U^{-1}(u,w)$, we see that we can write $(u(x),w(x))=U(x)(z(x),v(x))$, with $(z(x),v(x))\in \mathbb{S}^1$, for any $x\in\R^n$. From
\[
\begin{split}
\int 9U(x)^4&=\int F(0,U(x))=\widetilde{N}(0,U)\\
& =\widetilde{N}(u,w)=\int F\Big(U(x)(z(x),v(x))\Big)=\int F(z(x),w(x))U(x)^4
\end{split}
\]
it follows that
$$
\int U(x)^4(9-F(z(x),w(x)))=0,
$$
Therefore, $F(z(x),v(x))=9$ for a.e. $x\in\R^n$, which implies that either $(z(x),v(x))=(0,1)$ or $(z(x),w(x))=(0,-1)$ for a.e. $x\in R^n$. Consequently, $(u(x),w(x))=(0,U(x))$ or $(u(x),w(x))=(0,-U(x))$, which is the desired conclusion.
\end{proof}

\begin{remark}\label{remb}
	In the case $\omega+1=\mu+3\sigma\omega$, besides the solutions of the form $(0,w)$ (with $w$ a solution of \eqref{nlsg}),  \eqref{standindsys} has another interesting solution. Indeed, assume that $Q=bP$, where $b$ is the (negative) real solution of the equation
	$$
	2b^2+\frac{1}{3}b+\frac{1}{9}=\frac{1}{b}\left( 9b^3+2b+\frac{1}{9} \right).
	$$
Then, equations in \eqref{standindsys} reduce to the same one, namely,
\begin{equation} \label{explP}
\Delta P-(\mu+3\sigma\omega)P+\left(	2b^2+\frac{1}{3}b+\frac{1}{9}\right)P^3=0.
\end{equation}
Hence, if $P_b$ is a solution of \eqref{explP} it follows that $(P_b,bP_b)$ is a solution \eqref{standindsys}. Note that, according to Proposition \ref{carground}, even if $P_b$ is a ground state of \eqref{explP} (which clearly exist), $(P_b,bP_b)$ is not a ground state of \eqref{standindsys}.

In the case $n=1$, the unique ground state of \eqref{nlsg} is explicitly given by
\begin{equation} \label{explsol}
w(x)=\frac{1}{3}\sqrt{2(\mu+3\sigma\omega)}\,{\rm sech}(\sqrt{(\mu+3\sigma\omega)}x).
\end{equation}
So, according to Proposition \ref{carground},  the unique ground sate of \eqref{standindsys} is $(0,w)$, with $w$ given in \eqref{explsol}.
\end{remark}

\section{Global well-posedness}

In this section we are interested in the study of the Cauchy problem associated with \eqref{nlssystem} in the energy space; so, we couple \eqref{nlssystem} with an initial  data $(u_0,w_0)$ in $H^1(\R^n)\times H^1(\R^n)$ and consider the problem
\begin{equation}\label{nlssystem1}
\begin{cases}
{\displaystyle iu_t+\Delta u-u+\left(\frac{1}{9}|u|^2+2|w|^2\right)u+\frac{1}{3}\overline{u}^2w=0},\\
{\displaystyle i\sigma w_t+\Delta w-\mu w+\Big(9|w|^2+2|u|^2\Big)w+\frac{1}{9}u^3=0},\\
u(x,0)=u_0(x), \quad w(x,0)=w_0(x).
\end{cases}
\end{equation}

By using the contraction mapping principle combined with the well-known Strichartz estimates, one can easily show the local well-posedness of \eqref{nlssystem1} (see \cite{Cazenave} or \cite{Linares} for details). More precisely, one may establish the following result:

\begin{theorem}\label{localtheorem}
Assume $1\leq n\leq 3$ and $u_0,w_0\in H^1(\R^n)$. Then, the Cauchy problem \eqref{nlssystem1} admits a unique solution,
$$
(u,w)\in C((-T_*,T^*); H^1(\R^n)\times H^1(\R^n))
$$
defined in the maximal interval of existence $(-T_*,T^*)$, where $T_*,T^*>0$. 

\medskip

\noindent
In addition, the following blow-up alternative holds: if $T^*<\infty$ then
$$
\lim_{t\to T^*}\left(\|\nabla u\|_{L^2}^2+\|\nabla w\|_{L^2}^2 \right)=+\infty.
$$
A similar statement holds with $T_*$ instead of $T^*$.
\end{theorem}

Since the quantity $M$ defined in \eqref{mass} is conserved and is equivalent to the standard norm of norm in $L^2\times L^2$, in order to prove the global well-posedness of \eqref{nlssystem1} in $H^1\times H^1$, one only needs to get an {\it a priori} bound on the $L^2$-norm of the gradients of $u$ and $w$. With this in mind, let us recall the functional
\begin{equation}
K(u,w)=\|\nabla u\|_{L^2}^2+\|\nabla w\|_{L^2}^2.
\end{equation}

To obtain an upper bound for $K$, we may use the conservation of the energy and H\"older's inequality combined with the Gagliardo-Nirenberg inequality
$$
\|f\|_{L^4}^4\leq C\|\nabla f\|_{L^2}^n\|f\|_{L^2}^{4-n}:
$$
\begin{equation}\label{gal}
\begin{split}
K(u,w)&\leq K(u,w)+\|u\|_{L^2}^2+\mu\|w\|_{L^2}^2\\
&= 2 E(u_0,v_0)+2\int\left(\frac{1}{36}|u|^4+\frac{9}{4}| w|^4+|u|^2| w|^2+\frac{1}{9}\Re e( \overline{u}^3w)\right)\\
& \leq 2 E(u_0,v_0)+2\int\left(\frac{1}{36}|u|^4+\frac{9}{4}| w|^4+|u|^2| w|^2+\frac{1}{9}|u|^3|w|\right)\\
&\leq 2 E(u_0,v_0)+2 C\left(\|u\|_{L^4}^4+\|w\|_{L^4}^4\right)\\
&\leq 2 E(u_0,v_0)+2C\left( \|\nabla u\|_{L^2}^2+\|\nabla w\|_{L^2}^2\right)^{n/2} \left( \| u\|_{L^2}^2+3\sigma\| w\|_{L^2}^2\right)^{2-n/2}\\
&=2 E(u_0,v_0)+2CK(u,w)^{n/2} M(u_0,w_0)^{2-n/2},
\end{split}
\end{equation}
where $C$ is a positive universal constant. An immediate consequence of \eqref{gal} is that if $n=1$ then $K(u(t),w(t))$ is bounded. Indeed, for all $\epsilon>0$,
$$K(u,w)\leq 2E(u_0,w_0)+C\Big(\epsilon K(u,w))+\frac 1{\epsilon}M(u_0,w_0)^3\Big),$$
and, choosing $\epsilon= 1/{2C}$,
$$K(u,w)\leq 4E(u_0,w_0)+4C^2M(u_0,w_0)^3.$$
In view of the blow-up alternative stated in Theorem \ref{localtheorem}, this yields the following corollary:
\begin{corollary}
	\label{cor32}
Assume $n=1$ and $u_0,w_0\in H^1(\R)$. Then, the Cauchy problem \eqref{nlssystem1} is globally well-posed.
\end{corollary}

Now, if $n=2$, \eqref{gal}  does not give an immediate {\it a priori} bound. However, in this case, we can rewrite it as
$$
(1-2CM(u_0,w_0))K(u,w)\leq 2E(u_0,w_0).
$$
Hence, if $M(u_0,w_0)<1/2C$  then the last inequality provides a bound for $K(u(t),w(t))$ and we deduce:

\begin{corollary}
	\label{cor33}
Assume $n=2$ and $u_0,w_0\in H^1(\R^2)$. Then, the Cauchy problem \eqref{nlssystem1} is globally well-posed, provided that the initial mass $M(u_0,w_0)$ is sufficiently small.
\end{corollary}

Next we focus on the question of how small $M(u_0,w_0)$ must be for the conclusion of Corollary \ref{cor33} to hold. As we observed above, the constant $C$ appearing in \eqref{gal} plays a crucial role in this question. So, in some sense, the problem is related with the best constant we can place in the inequality
\begin{equation}\label{gal1}
\int\left(\frac{1}{36}|u|^4+\frac{9}{4}| w|^4+|u|^2| w|^2+\frac{1}{9}|u|^3|w|\right)\leq CK(u,w)^{n/2} M(u,w)^{2-n/2}.
\end{equation}
Recall that for $u,w\in H^1$, 
\begin{equation*}
N(u,w):=\int\left(\frac{1}{36}u^4+\frac{9}{4} w^4+u^2 w^2+\frac{1}{9}{u}^3w\right).
\end{equation*}
Also, define
\begin{equation}\label{Jdef}
J(u,w):= \frac{K(u,w)^{n/2} M(u,w)^{2-n/2}}{N(u,w)}.
\end{equation}
It is easily seen that \eqref{gal1} is equivalent to
$$
\frac{1}{C}\leq J(u,w)
$$
for functions $(u,w)$ in the set
$$
\mathbf{N}:=\{(u,w)\in H^1(\R^n)\times H^1(\R^n); \; N(u,w)>0\}.
$$
In particular, the infimum of $J$ on $\mathbf{N}$ is clearly the reciprocal of the best constant in \eqref{gal1}.
In the sequel we will show that this infimum is indeed attained on $\mathbf{N}$. We start with the following preliminary result:

\begin{lemma}\label{proper}
Assume $1\leq n\leq 3$. Let $(P,Q)$ be any solution of \eqref{standindsys} with $\omega=0$ and $\mu=3\sigma$. Then,
\begin{equation}\label{po1}
N(P,Q)=S(P,Q),
\end{equation}
\begin{equation}\label{po2}
K(P,Q)=nS(P,Q),
\end{equation}
\begin{equation}\label{po3}
K(P,Q)=\frac{n}{4-n}M(P,Q).
\end{equation}
In particular,
\begin{equation}\label{po4}
J(P,Q)=n^{n/2}(4-n)^{2-n/2}S(P,Q).
\end{equation}
\end{lemma}
\begin{proof}
By summing \eqref{poha1} and \eqref{poha2} we promptly deduce that
\begin{equation}\label{po5}
K(P,Q)+M(P,Q)=4N(P,Q).
\end{equation}
Thus, since for $\omega=0$, $S=E$, we obtain
$$
S(P,Q)=E(P,Q)=\frac{1}{2}\Big(K(P,Q)+M(P,Q)\Big)-N(P,Q)=2N(P,Q)-N(P,Q)=N(P,Q),
$$
which proves \eqref{po1}. Also, the identity \eqref{po3} follows directly from \eqref{poha3}.\\
\\
Furthermore, from \eqref{po3}, $M(P,Q)+K(P,Q)=\frac 4n K(P,Q)$, and, from $$N(P,Q)=E(P,Q)=\frac 12(K(P,Q)+M(P,Q))-N(P,Q),$$
one obtains \eqref{po2}.\\
Finally, \eqref{po4} is a consequence of \eqref{po1}-\eqref{po3}. The proof of the lemma is thus completed.
\end{proof}

\begin{lemma}\label{proper1}
	Suppose $1\leq n\leq 3$.
The infimum of $J$ is attained on $\mathbf{N}$ at a pair of real functions $(P,Q)$, that is,
$$
\inf_{\mathbf{N}}J(u,w)=J(P,Q),
$$
if and only if, up to scaling, $(P,Q)$ is a ground state  solution  of \eqref{standindsys} with $\omega=0$ and $\mu=3\sigma$.
\end{lemma}
\begin{proof}
	Assume $(P,Q)$ is a minimum of $J$ on $\mathbf{N}$.
Since $(P,Q)$ is a critical point   we have $J'(P,Q)=0$, which implies that 
\begin{equation}\label{gal3}
\begin{cases}
{\displaystyle -\frac{n}{K(P,Q)}\Delta P+\frac{4-n}{M(P,Q)}P=\frac{1}{N(P,Q)}\left(\frac{1}{9}P^3+2Q^2P+\frac{1}{3}{P}^2Q\right),}\\\\
{\displaystyle -\frac{n}{K(P,Q)}\Delta Q + \frac{(4-n)3\sigma}{M(P,Q)} Q=\frac{1}{N(P,Q)}\left(9Q^3+2P^2Q+\frac{1}{9}P^3\right)}.
\end{cases}
\end{equation}
Now take $\lambda,\nu>0$ such that
$$
\lambda^2=\frac{nM(P,Q)}{(4-n)K(P,Q)} \quad \mbox{and} \quad \nu^2=\frac{M(P,Q)}{(4-n)N(P,Q)}
$$
and define
$$
\widetilde{P}(x)=\nu P(\lambda x), \qquad \widetilde{Q}(x)=\nu Q(\lambda x).
$$
A straightforward calculation reveals that $(\widetilde{P},\widetilde{Q})$ satisfies
\begin{equation}\label{gal4}
\begin{cases}
{\displaystyle -\Delta \widetilde{P}+\widetilde{P}=\left(\frac{1}{9}\widetilde{P}^3+2\widetilde{Q}^2\widetilde{P}+\frac{1}{3}{\widetilde{P}}^2\widetilde{Q})\right),}\\\\
{\displaystyle -\Delta \widetilde{Q}+ 3\sigma \widetilde{Q}=\left(9\widetilde{Q}^3+2\widetilde{P}^2\widetilde{Q}+\frac{1}{9}\widetilde{P}^3\right)},
\end{cases}
\end{equation}
which is exactly system \eqref{standindsys} with $\omega=0$ and $\mu=3\sigma$. In addition, it is not difficult to see that $J(\widetilde{P},\widetilde{Q})=J(P,Q)$ and $N(\widetilde{P},\widetilde{Q})=\nu^4\lambda^{-n}N(P,Q)>0$, which means that $(\widetilde{P},\widetilde{Q})$ is also a minimizer of $J$ on $\mathbf{N}$. Relation \eqref{po4} then yields that $(\widetilde{P},\widetilde{Q})$ is a minimizer of $S$ on $\mathbf{N}$. In view of \eqref{poha7}, it is easy to conclude that any bound state belongs to $\mathbf{N}$ and we deduce that $(\widetilde{P},\widetilde{Q})$ is a ground state.

\medskip

\noindent
Conversely, if $(P,Q)$ is a ground state  of \eqref{standindsys} with $\omega=0$ and $\mu=3\sigma$, we have $(P,Q)\in \mathbf{N}$ and $(P,Q)$ is a minimum of $S$. The identity \eqref{po4} again implies that $(P,Q)$ is also a minimum of $J$.
\end{proof}

\noindent
The above results allow us to obtain the best constant in the Gagliardo-Nirenberg inequality \eqref{gal1}. More precisely, we have:

\begin{corollary}\label{gnbest}
Assume $1\leq n\leq 3$. Then the inequality
$$
N(u,w)\leq C_{GN}K(u,w)^{n/2} M(u,w)^{2-n/2}
$$
holds, for any $(u,v)\in \mathbf{N}$, with
\[
\begin{split}
C_{GN}&=\frac{(4-n)^{n/2-1}}{n^{n/2}} \frac{1}{M(P,Q)}\\
&=\frac{(4-n)^{n/2-2}}{n^{n/2}} \frac{1}{S(P,Q)},
\end{split}
\]
where $(P,Q)$ is any ground state  of \eqref{standindsys} with $\omega=0$ and $\mu=3\sigma$.
\end{corollary}
\begin{proof}
It suffices to recall that 
$$
\frac{1}{C_{GN}}=\inf_{\mathbf{N}}J(u,w)
$$
and use Lemmas \ref{proper} and \ref{proper1}.
\end{proof}

\begin{remark}
Note that the constant $C_{GN}$ does not depend on the choice of the ground state $(P,Q)$ since all ground states have the same mass $M$ (and the same action $S$). Hence, the question of uniqueness of ground states is not an issue here.
\end{remark}

With Corollary \ref{gnbest} in hand we can to prove the following Theorem:

 \begin{theorem}\label{globaln=2}
	Assume $n=2$ and $u_0,w_0\in H^1(\R^2)$. Then the Cauchy problem \eqref{nlssystem1} is globally well-posed provided that 
	$$
	M(u_0,w_0)<M(P,Q),
	$$
	where $(P,Q)$ is any ground state  of \eqref{standindsys} with $\omega=0$ and $\mu=3\sigma$.
\end{theorem}
\begin{proof}
Indeed, It suffices to use \eqref{gal} with the constant $C$ replaced by $C_{GN}$ given in Corollary \ref{gnbest}.
\end{proof}

\bigskip

Next we turn attention to the global well-posedness for $n=3$. We begin by stating the following Lemma, whose proof can be found in \cite{beg}  and \cite{Pastor}:
\begin{lemma}\label{supercase}
	Let $I$ be an open interval with $0\in I$. Let $a\in \R$, $b>0$ and $q>1$. Define $\gamma=(bq)^{-\frac{1}{q-1}}$ and $f(r)=a-r+br^{q}$, for $r\geq 0$. Let $G(t)$ be a nonnegative continuous  function such that $f\circ G\geq 0$ on $I$. Assume that $a<\left(1-\frac{1}{q}\right)\gamma$.
	\begin{enumerate}
		\item[(i)] If $G(0)<\gamma$, then $G(t)<\gamma$, $\forall t\in I$.
		\item[(ii)] If $G(0)>\gamma$, then $G(t)>\gamma$, $\forall t\in I$.
	\end{enumerate}
In addition if $a<(1-\delta_1)\left(1-\frac{1}{q}\right)\gamma$ and $G(0)>\gamma$, for some $\delta_1>0$, then there exists $\delta_2$, depending only on $\delta_1$ such that $G(t)>(1+\delta_2)\gamma$, $\forall t\in I$.
\end{lemma}

Our main theorem here reads as follows.

\begin{theorem}\label{globaln=3}
	Assume $n=3$ and $u_0,w_0\in H^1(\R^3)$. Suppose that
	\begin{equation}\label{absass1}
	E(u_0,w_0)M(u_0,w_0)<\frac{1}{2}E(P,Q)M(P,Q)
	\end{equation}
	and
	\begin{equation}\label{absass2}
	K(u_0,w_0)M(u_0,w_0)<K(P,Q)M(P,Q),
	\end{equation}
	where $(P,Q)$ is any ground state  of \eqref{standindsys} with $\omega=0$ and $\mu=3\sigma$. Then, as long as the local solution given in Theorem \ref{localtheorem} exists, there holds
	\begin{equation}\label{globalcon1}
	K(u(t),w(t))M(u(t),w(t))<K(P,Q)M(P,Q).
	\end{equation}
	In particular, this implies that the Cauchy problem \eqref{nlssystem1} is globally well-posed under conditions \eqref{absass1} and \eqref{absass2}.
\end{theorem}
\begin{proof} 
Let $a=2E(u_0,w_0)$,  $ b=2C_{GN}M(u_0,w_0)^{1/2}$, and $q=3/2$. If $G(t)=K(u(t),w(t))$, from \eqref{gal}, with $C_{GN}$ instead of $C$, we obtain $f\circ G\geq0$, where $f(r)=a-r+br^{3/2}$. Also, by using Lemma \ref{proper} we see that
$$
\gamma=\frac{3M(P,Q)^2}{M(u_0,w_0)}.
$$
In addition, a simple calculation using Lemma \ref{proper} also reveals that
$$
a<\left(1-\frac{1}{q}\right)\gamma \quad \Leftrightarrow \quad E(u_0,w_0)M(u_0,w_0)<\frac{1}{2}E(P,Q)M(P,Q)
$$
and
$$
G(0)<\gamma \quad \Leftrightarrow \quad K(u_0,w_0)M(u_0,w_0)<K(P,Q)M(P,Q).
$$
Hence, Lemma \ref{supercase} implies that \eqref{globalcon1} holds. This completes the proof of the theorem.
\end{proof}

\section{Blow up}

In this section we will show some blow up results.

\begin{definition}
We say that the solution of \eqref{nlssystem1}, given in Theorem \ref{localtheorem}, blows up forward in time if $T^*<\infty$ and backward in time if $T_*<\infty$. We say that the solution blows up if it blows up forward and backward in time.
\end{definition}

Our results of this Section will show that the condition in Theorem \ref{globaln=2} is sharp, at least for some parameters $\sigma$ and $\mu$. Actually, in the case $n=2$ we can construct an explicit solution that blows up, say, forward in time.

\begin{theorem}\label{sharpn=2}
 Assume $n=2$, $\sigma=3$, and $\mu=9$. Let $(P,Q)$ be any ground state  of \eqref{standindsys} with $\omega=0$ (and $\mu=3\sigma$). Then, there exists $u_0,w_0\in H^1$ satisfying $M(u_0,w_0)=M(P,Q)$ such that the corresponding solution of the Cauchy problem \eqref{nlssystem1} blows up  forward in time.
\end{theorem}
\begin{proof}
First we note that $(u,w)$ is a solution of \eqref{nlssystem1} if and only if 
$$
\widetilde{u}(x,t)=e^{it}u(x,t), \qquad \widetilde{w}(x,t)=e^{3it}w(x,t)
$$
is a solution of 
\begin{equation}\label{nlssystem2}
\begin{cases}
i\widetilde{u}_t+\Delta \widetilde{u}+(\frac{1}{9}|\widetilde{u}|^2+2|\widetilde{w}|^2)\widetilde{u}+\frac{1}{3}\overline{\widetilde{u}}^2\widetilde{w}=0,\\
i\sigma \widetilde{w}_t+\Delta \widetilde{w}+(9|\widetilde{w}|^2+2|\widetilde{u}|^2)\widetilde{w}+\frac{1}{9}\widetilde{u}^3=0,\\
\widetilde{u}(x,0)=u_0(x), \quad \widetilde{w}(x,0)=w_0(x).
\end{cases}
\end{equation}
Actually, this equivalence is true only under the condition $\mu=3\sigma$. So the problem is reduced to showing that \eqref{nlssystem2} has a solution with  $M(u_0,w_0)=M(P,Q)$ that blows up forward in time.

Next, a tedious but straightforward calculation gives that if $(\widetilde{u},\widetilde{w})$ is a solution of the differential equations in \eqref{nlssystem2} so is the pair $(\widehat{u},\widehat{w})$ defined by
$$
\widehat{u}(x,t)=\frac{1}{1-t}e^{-\frac{i|x|^2}{4(1-t)}}\widetilde{u}\left(\frac{x}{1-t},\frac{t}{1-t}\right), \quad \widehat{w}(x,t)=\frac{1}{1-t}e^{-\frac{3i|x|^2}{4(1-t)}}\widetilde{w}\left(\frac{x}{1-t},\frac{t}{1-t}\right).
$$
In addition,
$$
\widehat{u}(x,0)=e^{-\frac{i|x|^2}{4}}u_0(x), \qquad \widehat{u}(x,0)=e^{-\frac{3i|x|^2}{4}}w_0(x).
$$
Finally, by taking
$$
\widetilde{u}(x,t)=e^{it}P(x), \qquad \widetilde{w}(x,t)=e^{3it}Q(x),
$$ 
it is easily  seen that $(\widetilde{u},\widetilde{w})$ is a solution of the equations in \eqref{nlssystem2}. Consequently, 
$$
\widehat{u}(x,t)=\frac{1}{1-t}e^{-\frac{i|x|^2}{4(1-t)}}e^{\frac{it}{1-t}}P\left(\frac{x}{1-t}\right), \quad \widehat{w}(x,t)=\frac{1}{1-t}e^{-\frac{3i|x|^2}{4(1-t)}}e^{\frac{3it}{1-t}}Q\left(\frac{x}{1-t}\right)
$$
is a solution of \eqref{nlssystem2} that blows up at time $t=1$ and satisfies $M(\widehat{u}(0),\widehat{w}(0))=M(P,Q)$.
\end{proof}

\begin{remark}
By using the same ideas as in the proof of Theorem \ref{sharpn=2} one can construct a blowing up solution at any time $T\neq0$. In particular, we can also construct a solution that blows up backward in time.
\end{remark}

The Theorem \ref{sharpn=2} holds only in dimension $d=2$, the critical dimension.
Next we will obtain some virial identities to system \eqref{nlssystem}. First observe
 that \eqref{nlssystem} can be written in the pseudo-Hamiltonian form
\begin{equation}
\frac{d}{dt}X(t)=\Lambda E'(X(t)),
\end{equation}
where $X(t)=(u(t),w(t))$, $E'$ stands for the Fr\'echet derivative of $E$, and $\Lambda$ is the skew-adjoint operator given by
\begin{equation}
\Lambda=\left(
\begin{array}{cc}
-i & 0\\
0 & - i/\sigma
\end{array}
\right).
\end{equation}

\begin{proposition}\label{virialprop}
Assume
$$
u_0,w_0\in H^1(\R^n)\cap L^2(\R^n,|x|^2dx)=:\mathbb{H}
$$
and define
$$
V(t)=\int |x|^2(|u(t)|^2+3\sigma|w(t)|^2),
$$
where $(u(t),w(t))$ is the maximal solution of \eqref{nlssystem1}, with initial data $(u_0,w_0)$, and defined in the maximal time interval $[0,T^*)$. Then $V\in C^2\left([0,T^*)\right)$. In addition,
\begin{equation}\label{vlinha}
V'(t)=4Im \int \left(\overline{u}(t)x\cdot\nabla u(t)+3 \overline{w}(t)x\cdot\nabla w(t)\right)
\end{equation}
and 
\begin{equation}\label{vdoislinha}
\begin{split}
V''(t)&=\int\left(8|\nabla u|^2 +8|\nabla w|^2- \frac{2n}{9}|u|^4-\frac{54n}{\sigma}| w|^4-8n|u|^2|w|^2\right)\\
&\quad +2\left(\frac{24}{\sigma}-8\right)\Re e\int \overline{u}|w|^2x\cdot\nabla u +\frac{1}{9}\left(\frac{12}{\sigma}-12\right)n \Re e\int \overline{u}^3 w\\
&\quad +\frac{1}{9}\left(\frac{24}{\sigma}-8\right) \Re e\int 3\overline{u}^2 wx\cdot\nabla u.
\end{split}
\end{equation}
\end{proposition}
\begin{proof}
We proceed formally. Introduce the functional
$$
\V(u,w)=\int |x|^2(|u|^2+3\sigma|w|^2)
$$
and note that $V(t)=\V(u(t),w(t))\equiv \V(X(t))$. Thus,
\begin{equation}
V'(t)=\frac{d}{dt}\V(X(t))=\langle \V'(X(t)),\frac{d}{dt}X(t)\rangle=\langle \V'(X(t)),JE'(X(t))\rangle=:P(X(t)).
\end{equation}
Thus, in order to determine $V'(t)$, it suffices to determine the functional $P$. To do so, we use a dual Hamiltonian system. Indeed, given $Y_0=(\widetilde{u}_0,\widetilde{w}_0)\in \mathbb{H}$, assume the initial-value problem
\begin{equation}\label{auxHam}
\frac{d}{dt}Y(t)=\Lambda\V'(Y(t)), \qquad Y(0)=Y_0
\end{equation}
is (at least) locally well-posed. Then
\begin{equation}
\frac{d}{dt}E(Y(t))=\langle E'(Y(t)),\frac{d}{dt}Y(t)\rangle=\langle E'(Y(t)),\Lambda\V'(Y(t))\rangle =-\langle \V'(Y(t)),\Lambda E'(Y(t))\rangle=-P(Y(t)).
\end{equation}
Evaluating at $t=0$, we deduce
$$
P(Y_0)=-\frac{d}{dt}E(Y(t))\Big|_{t=0}.
$$
In conclusion, in order to determine the first derivative of $V(t)$, it suffices to solve \eqref{auxHam} and then take the derivative of the energy at this solution evaluated at $t=0$.

Next we solve \eqref{auxHam}. Indeed, if $Y(t)=(\widetilde{u}(t),\widetilde{w}(t))$, it easy to see that \eqref{auxHam} is equivalent to
\[
\begin{cases}
\dfrac{d}{dt}(\widetilde{u}(t),\widetilde{w}(t))=(-2i|x|^2\widetilde{u},-6i|x|^2\widetilde{w})\\
\widetilde{u}(0)=\widetilde{u}_0, \;\;\widetilde{w}(0)=\widetilde{w}_0,
\end{cases}
\]
whose solution is
$$
Y(t)=(\widetilde{u}(t),\widetilde{w}(t))=(e^{-2i|x|^2t}\widetilde{u}_0, e^{-6i|x|^2t}\widetilde{w}_0).
$$
Hence,
\[
\begin{split}
P(Y_0)&=-\frac{d}{dt}E(Y(t))\Big|_{t=0}=\frac{1}{2}\left(\int |\nabla \widetilde{u}(t)|^2+ |\nabla \widetilde{w}(t)|^2 \right)\Big|_{t=0}\\
&= 4Im \int \left( \overline{\widetilde{u}}_0x\cdot\nabla \widetilde{u}_0 +3\overline{\widetilde{w}}_0x\cdot\nabla \widetilde{w}_0 \right).
\end{split}
\]
This establishes \eqref{vlinha}.

To compute $V''(t)$ we use the above argument replacing $V(t)$ by $V'(t)$ and $\V(u,w)$ by
$$
G(u,w)=4Im \int \left(\overline{u}x\cdot\nabla u+3 \overline{w}x\cdot\nabla w\right).
$$ 
Since 
$$
G'(u,w)=-4i\left(2x\cdot\nabla u+nu, 6x\cdot\nabla w+3nw\right),
$$
we see that
\begin{equation}\label{auxHam1}
\frac{d}{dt}Y(t)=JG'(Y(t)), \qquad Y(0)=Y_0
\end{equation}
is equivalent to
\[
\begin{cases}
\dfrac{d}{dt}(\widetilde{u}(t),\widetilde{w}(t))=(-8x\cdot\nabla \widetilde{u}-4n\widetilde{u},-\frac{24}{\sigma}x\cdot\nabla \widetilde{w}-\frac{12n}{\sigma}\widetilde{w})\\
\widetilde{u}(0)=\widetilde{u}_0, \;\;\widetilde{w}(0)=\widetilde{w}_0,
\end{cases}
\]
It is not difficult to check that the solution of the above initial-value problem is
$$
Y(t)=(\widetilde{u}(t),\widetilde{w}(t))=(e^{-4nt}\widetilde{u}_0(e^{-8t}x), e^{-\frac{12}{\sigma}t}\widetilde{w}_0(e^{-\frac{24}{\sigma}t}x)).
$$
Hence,
\[
\begin{split}
E(Y(t))&=\int\left(\frac{1}{2}e^{-16t}|\nabla \widetilde{u}_0|^2 +\frac{1}{2}e^{-\frac{48}{\sigma}t}|\nabla \widetilde{w}_0|^2- \frac{1}{36}e^{-8nt}|\widetilde{u}_0|^4- \frac{9}{4}e^{-\frac{24}{\sigma}nt}| \widetilde{w}_0|^4+\frac{1}{2} |\widetilde{w}_0|^2+\frac{\mu}{2}|\widetilde{w}_0|^2\right)\\
&\quad -e^{-8nt}\int|\widetilde{u}_0(e^{\left(\frac{24}{\sigma}-8\right)t}x) |^2|\widetilde{w}_0(x)|^2dx -\frac{1}{9}e^{\left( \frac{12}{\sigma}-12\right)nt}Re\int \overline{\widetilde{u}}_0^3(e^{\left(\frac{24}{\sigma}-8\right)t}x) \widetilde{w}_0(x)dx
\end{split}
\]
and
\[
\begin{split}
\frac{d}{dt}E(Y(t))\Big|_{t=0}&=\int\left(-\frac{16}{2}|\nabla \widetilde{u}_0|^2 -\frac{48}{2\sigma}|\nabla \widetilde{w}_0|^2+ \frac{8n}{36}|\widetilde{u}_0|^4+ \frac{9}{4}\frac{24n}{\sigma}| \widetilde{w}_0|^4+8n|\widetilde{u}_0|^2|\widetilde{w}_0|^2\right)\\
&\quad -2\left(\frac{24}{\sigma}-8\right)Re\int \overline{\widetilde{u}}_0|\widetilde{w}_0|^2x\cdot\nabla \widetilde{u}_0 -\frac{1}{9}\left(\frac{12}{\sigma}-12\right)n Re\int \overline{\widetilde{u}}_0^3 \widetilde{w}_0\\
&\quad -\frac{1}{9}\left(\frac{24}{\sigma}-8\right) Re\int 3\overline{\widetilde{u}}_0^2 \widetilde{w}_0x\cdot\nabla\widetilde{u}_0.
\end{split}
\]
Consequently, by recalling that $V''(t)$ must be the  above expression (with the opposite sign) when we replace $(\widetilde{u}_0,\widetilde{w}_0)$ by $(u(t),w(t))$, \eqref{vdoislinha} follows. The proof of the proposition is thus completed.
\end{proof}

\begin{corollary}\label{corviri}
Under the assumptions of Proposition \ref{virialprop}, if $\sigma=3$ then
$$
V''(t)=8nE(u_0,w_0)+4(2-n)\int(|\nabla u|^2+|\nabla w|^2)-4n\int (|u|^2+\mu|w|^2)
$$
\end{corollary}
\begin{proof}
It follows easily from Proposition \ref{virialprop}. Indeed, a simple computation yields
\[
\begin{split}
V''(t)=16E(u_0,w_0)+8(2-n)\int \left(\frac{1}{36}|u|^4+\frac{9}{4} |w|^4+|uw|^2+\frac{1}{9}Re\overline{u}^3w\right)-8\int (|u|^2+\mu|w|^2),
\end{split}
\]
and, by the definition of the energy functional,
\[
\begin{split}
V''(t)&=16E(u_0,w_0)+8(2-n)\left[ \frac{1}{2}\int \left(|\nabla u|^2+|\nabla w|^2+|u|^2+\mu|w|^2\right)-E(u_0,v_0)\right]\\
&\quad-8\int (|u|^2+\mu|w|^2)\\
&=8nE(u_0,w_0)+4(2-n)\int(|\nabla u|^2+|\nabla w|^2)-4n\int (|u|^2+\mu|w|^2),
\end{split}
\]
as claimed.
\end{proof}

With Corollary \ref{corviri} in hand we can also show that, under the assumption \eqref{absass1}, the condition \eqref{absass2} is sharp (at least in the case $\sigma=3$ and $\mu=9$) to obtain the global well posedness of \eqref{nlssystem1}. More precisely, we have

\begin{theorem}\label{sharpn=3}
	Assume $n=3$, $\sigma=3$, $\mu=9$. Suppose that
	\begin{equation}\label{absass3}
	E(u_0,w_0)M(u_0,w_0)<\frac{1}{2}E(P,Q)M(P,Q)
	\end{equation}
	and
	\begin{equation}\label{absass4}
		K(u_0,w_0)M(u_0,w_0)>K(P,Q)M(P,Q),
	\end{equation}
	where $(P,Q)$ is any ground state  of \eqref{standindsys} with $\omega=0$ (and $\mu=3\sigma$). Then, as long as the local solution given in Theorem \ref{localtheorem} exist there holds
	\begin{equation}\label{sharpcon1}
	K(u(t),w(t))M(u(t),w(t))>K(P,Q)M(P,Q).
	\end{equation}
	In particular, if $u_0,w_0\in\mathbb{H}$
	then the solution blows up in time.
\end{theorem}
\begin{proof}
In view of  (ii) in Lemma \ref{supercase}, the proof of the first part is similar to the one of the Theorem \ref{globaln=3}; so we omit the details.

Assume now $u_0,w_0\in \mathbb{H}$. From  assumption \eqref{absass3}, we can find a sufficiently small $\delta_1>0$ satisfying
	$$
E(u_0,w_0)M(u_0,w_0)<\frac{1}{2}(1-\delta_1)E(P,Q)M(P,Q).
$$
Consequently, using Lemma \ref{supercase}, there exists $\delta_2>0$ (depending only on $\delta_1$) such that
$$
	K(u(t),w(t))M(u_0,w_0)>(1+\delta_2)K(P,Q)M(P,Q).
$$
 Thus, from Corollary \ref{corviri}, we deduce that
\[
\begin{split}
V''(t)&<24E(u_0,w_0)M(u_0,w_0)-4K(u(t),w(t))M(u_0,w_0)\\
&<12(1-\delta_1)E(P,Q)M(P,Q)-4(1+\delta_2)K(P,Q)M(P,Q)\\
&=4(1-\delta_1)K(P,Q)M(P,Q)-4(1+\delta_2)K(P,Q)M(P,Q)\\
&=-4(\delta_1+\delta_2)K(P,Q)M(P,Q),
\end{split}
\]
where we have used that $K(P,Q)=3E(P,Q)$.
Since the right-hand side of this last inequality is negative, a standard convexity argument allows us to conclude.
\end{proof}

Next, we state some sufficient conditions which imply that the solution blows up either forward or backward in time.

\begin{theorem}
	\label{T47}
	Assume $2\leq n\leq3$, $\sigma=3$ and $\mu>0$. Suppose $u_0,w_0\in\mathbb{H}$
	and let
	$$
	(u,v)\in C((-T_*,T^*); \mathbb{H}\times \mathbb{H} )
	$$
	be the maximal solution of \eqref{nlssystem1} given in Theorem \ref{localtheorem}.
	The following statements hold:\\
	\begin{itemize}
		\item[(i)] If $E(u_0,w_0)<0$ then $T_*<\infty$ and $T^*<\infty$.
		\item[(ii)] If $E(u_0,w_0)=0$  and
		$$
		Im \int \left(\overline{u}_0x\cdot\nabla u_0+3 \overline{w}_0x\cdot\nabla w_0\right)<0,	
		$$
		then $T^*<\infty$.
		\item[(iii)] If $E(u_0,w_0)=0$  and
		$$
		Im \int \left(\overline{u}_0x\cdot\nabla u_0+3 \overline{w}_0x\cdot\nabla w_0\right)>0,	
		$$
		then $T_*<\infty$.
		\item[(iv)] If $E(u_0,w_0)>0$  and
		$$
		Im \int \left(\overline{u}_0x\cdot\nabla u_0+3 \overline{w}_0x\cdot\nabla w_0\right)<-\sqrt{nE(u_0,w_0)M(xu_0,xw_0)}	
		$$
		then $T^*<\infty$.
		\item[(v)] If $E(u_0,w_0)>0$  and
		$$
		Im \int \left(\overline{u}_0x\cdot\nabla u_0+3 \overline{w}_0x\cdot\nabla w_0\right)>\sqrt{nE(u_0,w_0)M(xu_0,xw_0)}	
		$$
		then $T_*<\infty$.
	\end{itemize}
\end{theorem}
\begin{proof}
	It is clear from Corollary \ref{corviri} that $V''(t)\leq 8nE(u_0,w_0)$. So, the proof follows the standard convexity method and we shall omit the calculations. The interested reader will find the details for the classical Schr\"odinger equation in \cite[Section 6.5]{Cazenave}.
\end{proof}

In the particular case $\mu=9$, the above result can be improved in the following sense:

\begin{theorem}\label{blow2}\label{T48}
	Assume $2\leq n\leq3$, $\sigma=3$ and $\mu=9$. Suppose
	$
	u_0,w_0\in \mathbb{H}
	$
	and let
	$$
	(u,v)\in C((-T_*,T^*); \mathbb{H}\times \mathbb{H} )
	$$
	be the maximal solution of \eqref{nlssystem1} given in Theorem \ref{localtheorem}.
	Then,\\
	\begin{itemize}
		\item[(i)] If $2E(u_0,w_0)<M(u_0,w_0)$ then $T_*<\infty$ and $T^*<\infty$.
		\item[(ii)] If $2E(u_0,w_0)=M(u_0,w_0)$  and
		$$
		Im \int \left(\overline{u}_0x\cdot\nabla u_0+3 \overline{w}_0x\cdot\nabla w_0\right)<0,	
		$$
		then $T^*<\infty$.
		\item[(iii)] If $2E(u_0,w_0)=M(u_0,w_0)$  and
		$$
		Im \int \left(\overline{u}_0x\cdot\nabla u_0+3 \overline{w}_0x\cdot\nabla w_0\right)>0,	
		$$
		then $T_*<\infty$.
		\item[(iv)] If $2E(u_0,w_0)>M(u_0,w_0)$  and
		$$
		\sqrt{2}Im \int \left(\overline{u}_0x\cdot\nabla u_0+3 \overline{w}_0x\cdot\nabla w_0\right)<-\sqrt{n(2E(u_0,w_0)-M(u_0,w_0))M(xu_0,xw_0)}	
		$$
		then $T^*<\infty$.
		\item[(v)] If $2E(u_0,w_0)>M(u_0,w_0)$  and
		$$
		\sqrt{2}Im \int \left(\overline{u}_0x\cdot\nabla u_0+3 \overline{w}_0x\cdot\nabla w_0\right)>\sqrt{n(2E(u_0,w_0)-M(u_0,w_0))M(xu_0,xw_0)}	
		$$
		then $T_*<\infty$.
	\end{itemize}
\end{theorem}
\begin{proof}
	In this case, the last integral in \ref{corviri} becomes $M(u_0,w_0)$. Hence,  $V''(t)\leq 4n(2E(u_0,w_0)-M(u_0,w_0))$.
\end{proof}

\begin{remark}
Under the assumption  $2E(u_0,w_0)<M(u_0,w_0)$ (and $\sigma=3$, $\mu=9$) a simple calculation using the definition of the energy and Lemma \ref{gnbest} shows that 
$$
K(u_0,w_0)^{n-2}M(u_0,w_0)^{4-n}>\frac{n^n}{4(4-n)^{n-2}}M(P,Q)^2.
$$
Hence, for $n=2$, we obtain $M(u_0,w_0)>M(P,Q)$, which does not contradict Theorem \ref{sharpn=2}. On the other hand, for $n=3$, using that $K(P,Q)=3M(P,Q)$,
$$
K(u_0,w_0)M(u_0,w_0)>\frac{9}{4}K(P,Q)M(P,Q),
$$
which implies that \eqref{absass4} holds.
\end{remark}

\section{Stability/Instability of ground states}

This section is devoted to study the (orbital) stability/instability of the standing waves \eqref{standing}  in some particular cases.
Let $(P,Q)$ be a real ground state of \eqref{standindsys}. In particular $Q\neq0$ and $(P,Q)$ must satisfy
\begin{equation}\label{standindsys1}
\begin{cases}
\Delta P-(\omega+1)P+(\frac{1}{9}P^2+2Q^2)P+\frac{1}{3}{P}^2Q=0,\\
\Delta Q-(\mu+3\sigma\omega) Q+(9Q^2+2P^2)Q+\frac{1}{9}P^3=0.
\end{cases}
\end{equation}

To start with, let us make clear our notion of stability and instability. Recall that \eqref{nlssystem} is invariant by translations and rotations, that is, if $(u,w)$ is a solution of \eqref{nlssystem} so are $(u(\cdot+y)w(\cdot+y))$ and $(e^{i\theta} u,e^{3i\theta}w)$, for any $\theta\in\R$ and $y\in\R^n$. Thus, the orbit generated by $(P,Q)$ is defined by
$$
\Omega=\{ (e^{i\theta}u(\cdot+y), e^{3i\theta}u(\cdot+y)):\;\; \theta\in\R,y\in\R^n \}.
$$

\begin{definition}[Orbital stability]\label{definstability}
We say that a standing wave $(e^{i\omega t}P,e^{3i\omega t}Q)$ is  orbitally stable by the flow of \eqref{nlssystem} if for any $\epsilon>0$ there exists a $\delta>0$ with the following property: if $(u_0,w_0)\in H^1\times H^1$ satisfies $\|(u_0,w_0)-(P,Q)\|_{H^1\times H^1}<\delta$ then the solution of \eqref{nlssystem}, with initial data $(u_0,w_0)$ is global and satisfies 
$$
\sup_{t\in\R}\inf_{(\theta,y)\in \R\times \R^n}\|(u(t),w(t))- (e^{i\theta}u(\cdot+y), e^{3i\theta}u(\cdot+y))\|_{H^1\times H^1}<\epsilon.
$$
Otherwise, we say that $(e^{i\omega t}P,e^{3i\omega t}Q)$ is  orbitally unstable by the flow of \eqref{nlssystem}.
\end{definition}

Roughly speaking, this means that there exists an $\epsilon$-neighborhood of $\Omega$ such that any solution of \eqref{nlssystem} starting in this neighborhood remains close to the orbit generated by $(P,Q)$. As usual in the current literature we say that $(P,Q)$ is orbitally stable (unstable) instead of saying that $(e^{i\omega t}P,e^{3i\omega t}Q)$ is  orbitally stable (unstable).

\subsection{Instability}

In order to establish our main theorem concerning instability let us introduce
\begin{equation}\label{manifold}
\Sigma:=\left\{(u,w)\in H^1(\R^n)\times H^1(\R^n):\;M(u,w)=M(P,Q)\right\}.
\end{equation}

Recall the following criterion for instability.

\begin{theorem}[Instability Criterion for ground states]\label{instbilitycrit}
Assume there exists $\Psi\in H^1(\R^n)\times H^1(\R^n)$ satisfying
\begin{itemize}
	\item[(i)] $\Psi$ belongs to the tangent spate $T_{(P,Q)}\Sigma$;
	\item[(ii)] $\Lambda^{-1}\Psi$ is $L^2$-orthogonal to $i(P,3 Q)$ and $\partial_{x_j}(P,Q)$, $j=1,\ldots,n$;
	\item[(iii)] $i(P,3 Q)$ and $\partial_{x_j}(P,Q)$, $j=1,\ldots,n$ are linearly independent;
	\item[(iv)] $\langle S''(P,Q)\Psi,\Psi\rangle<0$, where $S=E+\omega M$.
\end{itemize}
Then, $(P,Q)$ is orbitally unstable by the flow of \eqref{nlssystem}.
\end{theorem}
\begin{proof}
See \cite{Gon} and \cite{corcho}.
\end{proof}

We are now in position of proving the following result:

\begin{theorem}\label{instateo}
	Assume either $n=3$ and $\mu>0$ or $n=2$ and $\mu\neq3\sigma$. Let  $(P,Q)$ be a ground state. Then, the standing wave  $(e^{i\omega t}P,e^{3i\omega t}Q)$ is orbitally unstable by the flow of \eqref{nlssystem}.
\end{theorem}
\begin{proof}
We will check the assumptions in Theorem \ref{instbilitycrit}. To do so, let us introduce the smooth curve
$$
\Gamma(t)=\left(\gamma(t)\lambda^{\frac{n}{2}}(t)P(\lambda(t)\cdot), \alpha(t)\lambda^{\frac{n}{2}}(t)Q(\lambda(t)\cdot)\right),
$$
where $\alpha,\gamma$, and $\lambda$ are smooth functions to be choosing later satisfying,
\begin{eqnarray}\label{initial-con}
\alpha(0)=\gamma(0)=\lambda(0)=1.
\end{eqnarray}
In particular we have $\Gamma(0)=(P,Q)$.  Define the real number $k$ by
$$
k:=\dfrac{\int P^2}{3\sigma\int Q^2}.
$$
The assumption that $\Gamma(t)\subset\Sigma$ is equivalent to
\begin{equation}\label{krelation}
\gamma^2k+\alpha^2=k+1.
\end{equation}
So, from now on we will assume that \eqref{krelation} holds; so that once we choose the function
 $\alpha$, $\gamma$ is completely determined. By defining 
\begin{equation}
\Psi=\Gamma'(0)
\end{equation}
we promptly see that $\Psi\in T_{(P,Q)}\Sigma$; and condition (i) in Theorem \ref{instbilitycrit} holds.

Next we recall that
$$
\Lambda^{-1}=\left(
\begin{array}{cc}
i & 0\\
0 & \sigma i
\end{array}
\right).
$$
Hence $\Lambda^{-1}\Psi$ has complex components. This immediately implies that  $\Lambda^{-1}\Psi$ is orthogonal to $\partial_{x_j}(P,Q)$, $j=1,\ldots,n$. On the other hand, if $\Psi=(\Psi_1,\Psi_2)$, we have
$$
\Lambda^{-1}\Psi \perp i(P,3Q)\Leftrightarrow (\Psi_1,\sigma\Psi_2)\perp (P,3Q) \Leftrightarrow  (\Psi_1,\Psi_2)\perp (P,3\sigma Q)\Leftrightarrow \Psi \perp \nabla M(P,Q).
$$
Since $M(\Gamma(t))=M(P,Q)$, by taking the derivative with respect to $t$ and evaluating at $t=0$, it is clear that $\Psi \perp \nabla M(P,Q)$ and assumption (ii) in Theorem \ref{instbilitycrit} is checked.

Note that in (iii) there is nothing to check. So it remains to check (iv). To do so, first recall that
$S(\Gamma(t))=E(\Gamma(t))+\frac{\omega}{2} M(P,Q)$, because $\Gamma(t)\subset \Sigma$. Thus,
$$
\frac{d^2}{dt^2}E(\Gamma(t))=\frac{d^2}{dt^2}S(\Gamma(t))=\langle S''(\Gamma(t))\Gamma'(t),\Gamma'(t)\rangle+ \langle S'(\Gamma(t)),\Gamma''(t)\rangle.
$$
Evaluating at $t=0$ and using that $S'(P,Q)=0$, we see that (iv) is equivalent to 
\begin{equation}\label{Eder}
\frac{d^2}{dt^2}E(\Gamma(t))\Big|_{t=0}<0.
\end{equation}
Hence our task is to prove that we can choose $\alpha$ and $\lambda$ such that \eqref{Eder} holds. But, by using \eqref{krelation}, a simple calculation reveals that
$$
\frac{d}{dt}E(\Gamma(t))=\alpha'(t)A(t)+\lambda'(t)B(t)
$$
where
\[
\begin{split}
A(t)&=\int\left(-\frac{\alpha \lambda^2}{k}|\nabla P|^2+\alpha\lambda^2|\nabla Q|^2+\frac{1}{9k^2}(k+1-\alpha^2)\alpha\lambda^nP^4-9\alpha^3\lambda^nQ^4 \right)\\
&\quad +\int\left(\frac{2}{k}\alpha^3\lambda^nP^2Q^2-\frac{2}{k}(k+1-\alpha^2)\alpha\lambda^nP^2Q^2-\frac{\alpha}{k}P^2+\mu\alpha Q^2\right)\\
&\quad +\int\left( \frac{1}{3k^{3/2}}(k+1-\alpha^2)^{1/2}P^3Q-\frac{1}{9k^{3/2}}(k+1-\alpha^2)^{3/2}\lambda^nP^3Q  \right)
\end{split}
\]
and
\[
\begin{split}
B(t)&=\int\left(\frac{1}{k}(k+1-\alpha^2)\lambda|\nabla P|^2+\alpha^2\lambda|\nabla Q|^2-\frac{n}{36k^2}(k+1-\alpha^2)^2\lambda^{n-1}P^4-\frac{9n}{4}\alpha^4\lambda^{n-1}Q^4 \right)\\
&\quad +\int\left(-\frac{n}{k}(k+1-\alpha^2)\alpha^2\lambda^{n-1}P^2Q^2-\frac{n}{9k^{3/2}}(k+1-\alpha^2)^{3/2}\alpha\lambda^{n-1}P^3Q\right) 
\end{split}
\]
In view of \eqref{poha7} and \eqref{poha3}, we have
\[
\begin{split}
B(0)&=\int(|\nabla P|^2+|\nabla Q|^2)-\frac{n}{4}\int\left( \frac{1}{9}P^4+9Q^4+4P^2Q^2+\frac{4}{9}P^3Q \right)\\
&=-\frac{1}{2}\left((n-4)\int \left(|\nabla P|^2+|\nabla Q|^2\right) +n(\omega+1)\int P^2+n(\mu+3\sigma\omega)\int Q^2\right)\\
&=0.
\end{split}
\]
Also, in view of \eqref{poha2} and \eqref{poha1},
\[
\begin{split}
A(0)&=\int\left(-\frac{1}{k}|\nabla P|^2+|\nabla Q|^2+ \frac{1}{9k}P^4-9Q^4+\left(\frac{2}{k}-2\right)P^2Q^2-\frac{1}{k}P^2+\mu Q^2\right)\\
&\quad +\frac{1}{3k}\int P^3Q-\frac{1}{9}\int P^3Q \\
&=\int\left(-\frac{1}{k}|\nabla P|^2+|\nabla Q|^2+ \frac{1}{9k}P^4-9Q^4+\left(\frac{2}{k}-2\right)P^2Q^2-\frac{1}{k}P^2+\mu Q^2\right)\\
&\quad +\frac{1}{3k}\int P^3Q-\left(\int(|\nabla Q|^2+(\mu+3\sigma\omega)Q^2-9Q^4-2P^2Q^2)\right) \\
&=\frac{1}{k}\left(\int \left(-|\nabla P|^2-(\omega+1)P^2+\frac{1}{9}P^4+2P^2Q^2+\frac{1}{3}P^3Q \right)\right)\\
&=0.
\end{split}
\]
Therefore, by denoting $\alpha_0=\alpha'(0)$ and $\lambda_0=\lambda'(0)$, we deduce, after some calculations using Lemma \ref{pohojaevlemma},
\[
\begin{split}
&\frac{d^2}{dt^2}E(\Gamma(t))\Big|_{t=0}=\alpha_0A'(0)+\lambda_0B'(0)\\
&=\alpha_0^2\left[ \int\left( -\frac{2}{k^2}P^4+\frac{8}{k}P^2Q^2-18Q^4+\left(\frac{2}{3k}+\frac{1}{9}-\frac{1}{3k^2}\right)P^3Q \right)\right]\\
&\quad +2\alpha_0\lambda_0\left[ 2(3\sigma-\mu)\int Q^2+(n-2)\int\left(\frac{1}{9k}P^4-9Q^4+\left(\frac{2}{k}-2\right)P^2Q^2 +\left(\frac{1}{3k}-\frac{1}{9}\right)P^3Q \right) \right]\\
& \quad +\lambda_0^2\frac{n(2-n)}{4}\int\left(\frac{1}{9}P^4+9Q^4+4P^2Q^2+\frac{4}{9}P^3Q \right)\\
&\equiv A_0\alpha_0^2+2B_0\alpha_0\lambda_0+C_0\lambda_0^2.
\end{split}
\]
In particular, the second derivative of $E(\Gamma(t))$ at $t=0$ can be identified as a quadratic form associated with a symmetric matrix. Hence, it suffices to show that this quadratic form assume negative values. 

Assume first $n=2$. Then, it suffices to show that the discriminant
$$
D=A_0C_0-B_0^2=-\left(2(3\sigma-\mu)\int Q^2\right)^2
$$
is negative. But this statement is true provided $\mu\neq 3\sigma$.

Assume now $n=3$. By taking $(\alpha_0,\lambda_0)=(0,1)$ and using \eqref{poha7} and \eqref{poha3}, we obtain
\begin{equation}\label{Eder1}
\begin{split}
\frac{d^2}{dt^2}E(\Gamma(t))\Big|_{t=0}&=-\frac{3}{4}\int\left(\frac{1}{9}P^4+9Q^4+4P^2Q^2+\frac{4}{9}P^3Q \right)\\
&=-\int(|\nabla P|^2+|\nabla Q|^2),
\end{split}
\end{equation}
from which we deduce \eqref{Eder}. The proof of Theorem \ref{instateo} is thus completed.
\end{proof}

\subsection{Stability}

In this last section we study the orbital stability of the ground state given in Proposition \ref{carground}. First of all, we shall rewrite \eqref{nlssystem} as a real pseudo-Hamiltonian system in the form
\begin{equation*}
\frac{\partial X}{\partial t}(t)=\Lambda E'(X(t)),
\end{equation*}
where we have written $u=u_1+iu_2$, $w=w_1+iw_2$, $X=(u_1,w_1,u_1,u_2)$, $\Lambda$ is
the skew-symmetric linear operator defined by
\begin{equation}   \label{H}
\Lambda= \left( \begin{array}{cccc} 0 & 0 & 1 & 0    \\
0 & 0 & 0 & 1/\sigma   \\
-1 & 0 & 0 & 0    \\
0 & -1/\sigma & 0 & 0
\end{array}
\right)
\end{equation}
and $E$ is the energy function now given as
\begin{equation} \label{4.1}
\begin{split}    
E(u_1,u_w,u_1,w_2)&= \frac{1}{2}\int
\Big\{	|\nabla u_1|^2+|\nabla u_2|^2+|\nabla w_1|^2+|\nabla w_2|^2+u_1^2+u_2^2+\mu (w_1^2+w_2^2) 
  \\
& \quad\quad -\frac{1}{18}(u_1^4+2u_1^2u_2^2+u_2^4)-\frac{9}{2}(w_1^4+2w_1^2w_2^2+w_2^4) 
\\
& \quad \quad-2(u_1^2+u_2^2)(w_1^2+w_2^2)-\frac{2}{9}(u_1^3w_1+3u_1^2u_2w_2-3u_1u_2^2w_1-u_2^3w_2
\Big\}dx. 
\end{split}
\end{equation}

Our main theorem here reads as follows.

\begin{theorem}\label{teoesta}
Assume $n=1$ and $\omega+1=\mu+3\sigma\omega$. Let $(0,Q)$ be a ground state of \eqref{standindsys} according to Proposition \eqref{carground}. Then $(0,e^{3i\omega t}Q)$ is orbitally stable by the flow of \eqref{nlssystem}.
\end{theorem}

Here, if necessary, we will use $Q_\omega$ instead of $Q$ to emphasize that $Q$ depends on $\omega$. In addition, throughout the section, we assume $\omega+1=\mu+3\sigma\omega$.
In order to prove Theorem \ref{teoesta} we will use the well-known Grillakis, Shatah and Strauss' theory \cite{Grillakis4}. To simplify the notations, let $\Phi=(0, Q,0,0)$ and
\begin{equation} \label{4.2}
\mathcal{L_{\omega}}=S''(\Phi)= \left(
\begin{array}{cc}
\mathcal{L}_{R} & 0     \\
0 & \mathcal{L}_{I}
\end{array}
\right),
\end{equation}
where $\mathcal{L}_{R}$ and $\mathcal{L}_{I}$ are the $2\times 2$
matrix diagonal operators defined by
\begin{equation} \label{4.3}
\mathcal{L}_{R} = \left(
\begin{array}{cc}
-\Delta+(\omega+1)-2Q^2 & 0     \\
0 & -\Delta+(\alpha+3\sigma \omega)-27Q^2
\end{array}
\right)
\end{equation}
and
\begin{equation}  \label{4.4}
\mathcal{L}_{I} = \left(
\begin{array}{cc}
-\Delta+(\omega+1)-2Q^2 & 0     \\
0 & -\Delta+(\alpha+3\sigma \omega)-9Q^2
\end{array}
\right).
\end{equation}

In order to describe the spectrum of $\mathcal{L}_\omega$, we first
study the spectral properties of the following 
operators:
\begin{equation}  \label{4.5}
\mathcal{L}_{1} = -\Delta+(\alpha+3\sigma \omega)-27Q^2, \qquad \mathcal{L}_{2} =
-\Delta+(\alpha+3\sigma \omega)-9Q^2
\end{equation}
and
\begin{equation}  \label{4.6}
\mathcal{L}_{3} = -\Delta+(\omega+1)-2Q^2
\end{equation}

More precisely, we have:

\begin{theorem} \label{theorem4.1}
	Let $(0,Q)$ be as in Proposition \ref{carground}. Then:
	\begin{itemize}
		\item [(i)] The operator $\mathcal{L}_{1}$ in $(\ref{4.5})$ defined
		in $L^2(\R^n)$ has only one negative eigenvalue. Its kernel is given by
		$Ker(\mathcal{L}_1)=span\{Q_{x_i}; \;i=1,\ldots,n\}$ and the remainder
		of the spectrum is bounded away from zero.
		\item [(ii)] The operator $\mathcal{L}_{2}$ in $(\ref{4.5})$
		defined in $L^2(\R^n)$
		has no negative eigenvalues. Zero is a simple eigenvalue 
		with associated eigenfunction $Q$. Moreover, the remainder
		of the spectrum is bounded away from zero.
		\item [(iii)]The
		operator $\mathcal{L}_{3}$ in $(\ref{4.6})$ defined in $L^2(\R^n)$
		is a positive operator. Moreover, the remainder
		of the spectrum is bounded away from zero.
	\end{itemize}
\end{theorem}
\begin{proof}
These are well-known results, see for instance \cite{w1} and \cite{w2}. Note that (iii) is a consequence of (ii).
\end{proof}

As an immediate consequence, we have.

\begin{corollary}    \label{theorem4.2}
		Let $(0,Q)$ be as in Proposition \ref{carground}. Then the operator $\mathcal{L_{\omega}}$ has exactly one negative eigenvalue, 	$Ker(\mathcal{L_{\omega}})$ is $(n+1)$-dimensional and spanned by the set	$\{(0,0,0,Q),(0,Q_{x_i},0,0); \;i=1,\ldots,n\}$.  Moreover, the remainder
		of the spectrum is bounded away from zero.
\end{corollary}

Now we proof Theorem \ref{teoesta}.

\begin{proof}[Proof of Theorem \ref{teoesta}]
In view of Corollary \ref{theorem4.2} and the theory in \cite{Grillakis4} it suffices to prove that the second derivative of the function $d(\omega)=S(0,Q_\omega)$ is positive. But since $(0,Q_\omega)$ is a critical point of $S$ we have
$$
d'(\omega)=\frac{1}{2}M(0,Q_\omega)=\frac{3\sigma}{2}\int Q_\omega^2.
$$
Note that if $Q_0$ is the ground state of the equation
\begin{equation}\label{Qeq}
-\Delta Q+(\omega+1)Q-9Q^3=0,
\end{equation}
with $\omega=0$, then (by uniqueness)
$$
Q_\omega(x)=(\omega+1)^{1/2}Q_0\left((\omega+1)^{1/2}x\right)
$$
is the ground state of \eqref{Qeq} with $\omega>-1$. Thus,
$$
\int Q_\omega^2=\frac{1}{(\omega+1)^{n/2-1}}\int Q_0^2
$$
and
$$
d''(\omega)=\left(\frac{n}{2}-1\right)\frac{3\sigma}{2(\omega+1)^{n/2}}\int Q_0^2,
$$
from which we deduce $d''(\omega)>0$ for $n=1$. This completes the proof of the theorem.
\end{proof}

\section{Acknowledgment}
This work was developed in the frame of the CAPES-FCT convenium
Equa\c{c}\~oes de evolu\c{c}\~ao dispersivas. Ademir Pastor would like to thank the kind hospitality of Instituto Superior T\'ecnico, Universidade de Lisboa.
Filipe Oliveira was partially supported by the Project CEMAPRE - UID/
MULTI/00491/2013 financed by FCT/MCTES through national funds. Ademir Pastor was partially supported by CNPq/Brazil grants 402849/2016-7 and 303098/2016-3.

\section{Conflict of interest}
On behalf of all authors, the corresponding author states that there is no conflict of interest.

\small\noindent\textsc{Filipe Oliveira}\\
Mathematics Department and CEMAPRE\\
\noindent	
ISEG, Universidade de Lisboa\\
Rua do Quelhas 6, 1200-781 Lisboa, Portugal\\
\verb"foliveira@iseg.ulisboa.pt"

\bigskip

\noindent
\small\noindent\textsc{Ademir Pastor}\\
IMECC-UNICAMP,\\
Rua S\'ergio Buarque de Holanda, 651,\\
Cidade Universit\'aria, Campinas\\
SP 13083-859, Brazil\\
\verb"apastor@ime.unicamp.br"\\
\end{document}